\documentclass[11pt]{amsart}

\usepackage{fullpage}

\usepackage{amssymb,amsthm,amsmath,,amsrefs, amsfonts,amscd}
\usepackage[all]{xy}
\usepackage{dsfont}

\theoremstyle{definition}
\newtheorem{lemma}{Lemma}
\newtheorem{proposition}{Proposition}
\newtheorem{definition}{Definition}
\newtheorem{corollary}{Corollary}

\newtheorem{example}{Example}
\newtheorem{theorem}{Theorem}

\newcommand{\CC}{\mathrm{C}}
\newcommand{\K}{\mathcal{K}}

\newcommand{\N}{\mathbb{N}}
\newcommand{\Z}{\mathbb{Z}}
\newcommand{\R}{\mathbb{R}}
\newcommand{\C}{\mathbb{C}}
\newcommand{\Cu}{\mathrm{Cu}}

\newcommand{\Aut}{\mathrm{Aut}}

\newcommand{\id}{\mathrm{id}}

\newcommand{\Ad}{\mathrm{Ad}}
\newcommand{\Id}{\mathrm{Id}}
\newcommand{\M}{\mathrm{M}}
\newcommand{\dist}{\mathrm{dist}}
\newcommand{\IIm}{\mathrm{Im}}

\newcommand{\GL}{\mathrm{GL}}

\newcommand{\F}{\mathrm{F}}
\newcommand{\Ann}{\mathrm{Ann}}

\title[Crossed product by actions with the Rokhlin property]{Crossed product by actions of finite groups with the Rokhlin Property}

\author{Luis Santiago}\email{lsantiag@fields.toronto.edu}
\date{\today}

\address{Fields Institute for Research in Mathematical Sciences, 222 College St, Toronto, ON M5T 3J1, Canada}

\begin{document}
\maketitle

\begin{abstract}
We introduce and study a Rokhlin-type property for actions of finite groups on (not necessarily unital) C*-algebras. We show that the corresponding crossed product C*-algebras can be locally approximated by C*-algebras that are stably isomorphic to closed two-sided ideals of the given algebra. We then use this result to prove that several classes of C*-algebras are closed under crossed products by finite group actions with this Rokhlin property.
\end{abstract}

\tableofcontents

\section{Introduction}
The Rokhlin property is a form of freeness for group actions on noncommutative C*-algebras. For finite group actions on unital C*-algebras this property first appeared (under different names) in \cite{Kishimoto} and \cite{Fack-Marechal} for actions of $\mathbb{Z}_p$ on UHF-algebras, and in \cite{Herman-Jones} for actions of finite groups on arbitrary unital C*-algebras. 
Let us recall the definition of the Rokhlin property:
\begin{definition}(\cite[Definition 3.1]{Izumi-I})\label{def: Rokhlin unital}
Let $A$ be a unital C*-algebra and let $\alpha\colon G\to \Aut(A)$ be an action of a finite group $G$ on $A$. The action $\alpha$ is said to have the \emph{Rokhlin property} if there exists a partition of unity $(p_g)_{g\in G}\subset A_\infty$ (see Subsection \ref{central sequence algebra} for the definition of $A_\infty$) consisting of projections such that
$$\alpha_g(p_h)=p_{gh},$$
for all $g,h\in G$.
\end{definition}
When $A$ is separable this property can be reformulated in terms of finite subsets and projections of $A$ (\cite[Definition 1.1]{Phillips-TracialR}). The Rokhlin property is quite restrictive and is not useful if there is no abundance of projections. For instance, if either the $\mathrm{K}_0$-group or the $\mathrm K_1$-group of the algebra is isomorphic to $\Z$ then the algebra does not have any action of a finite group with the Rokhlin porperty (this follows by \cite[Theorem 3.13]{Izumi-I}). On the contrast, if a C*-algebra absorbs tensorially the UHF-algebra of infinite type $\M_{n^\infty}$ and $G$ is a finite group of order $n$ then the algebra has an action of $G$ with the Rokhlin property. In fact, in this case the actions of $G$ with the Rokhlin property are generic (\cite[Theorem 5.17]{Phillips-generic}).

For C*-algebras that are not unital there are at least two Rokhlin-type properties. One of them is called the multiplier Rokhlin property (\cite[Definition 2.15]{Phillips-freeness}), and is defined using projections of the multiplier algebra of the given C*-algebra. The other one, which is simply called the Rokhlin property (\cite[Definition 3.1]{Nawata}), is defined for $\sigma$-unital C*-algebras using projections in the central sequence algebra of the given C*-algebra. In this paper we introduce a Rokhlin-type property (Definition \ref{def: Rokhlin}) which somehow resembles \cite[Definition 1.1]{Phillips-TracialR}. It is defined using positive elements of the algebra instead of projections. This new Rokhlin-type property is essentially that of \cite[Definition 3.1]{Nawata} (they agree for separable C*-algebras), but it has the advantage that it can be defined for general C*-algebras (not just for $\sigma$-unital ones). Actions with the multiplier Rokhlin property or the Rokhlin property in the sense of \cite[Definition 3.1]{Nawata} also have this new Rokhlin property. Also, it can be seen that if the algebra is unital our definition agrees with the standard Rokhlin property for unital C*-algebras. 
In this paper we also study the crossed product C*-algebras obtained by actions of finite groups with the Rokhlin property on different classes of C*-algebras. Our main motivation comes from the study of crossed products by finite group actions on stably projectionless C*-algebras and our objective is to extend the results of \cite{Osaka-Phillips} to non-unital C*-algebras. The problem of classifying actions with the Rokhlin property will be considered in \cite{Gardella-Santiago}.

The paper is organized as follows. In Section 3 we define and study a Rokhlin-type property for actions of finite groups on arbitrary C*-algebras. Section 4 is dedicated to show that the crossed product C*-algebra by an action of a finite group on a C*-algebra with this Rokhlin property can be locally approximated by matrix algebras over hereditary C*-subalgebras of the given C*-algebra (Theorem \ref{thm: local approx}). This result may be considered as a nonunital version of \cite[Theorem 3.2]{Osaka-Phillips}. In Section 5 we use this approximation result to show that several classes of C*-algebras are closed under crossed products by finite group actions with the Rokhlin property (Theorem \ref{thm: permanence}). These classes of C*-algebras include purely infinite C*-algebras, C*-algebras of stable rank one, C*-algebras of real rank zero, C*-algebras of finite nuclear dimension, separable $\mathcal D$-stable C*-algebras, where $\mathcal D$ is a strongly self-absorbing C*-algebra, simple C*-algebras, simple C*-algebras with strict comparison of positive elements, simple stably projectionless C*-algebras, and C*-algebras that are stably isomorphic to sequential direct limits of one-dimensional NCCW-complexes.

I would like to thank Eusebio Gardella for useful discussions concerning the subject matter of this paper.

\section{Preliminary definitions and results}

\addtocontents{toc}{\protect\setcounter{tocdepth}{1}}
\subsection{Central sequence algebras}\label{central sequence algebra}

Let $A$ be a C*-algebra. Let us denote by $A^\infty$ the quotient of $\ell_\infty(\N, A)$ by the ideal $c_0(\N, A)$. That is, 
$$A^\infty:=\ell_\infty(\N, A)/c_0(\N, A).$$
The elements of $\ell_\infty(\N, A)$ will be written as sequences of elements of $A$. If $(a_n)_n\in \ell_\infty(\N, A)$ then its image in the quotient $A^\infty$ will be denoted by $[(a_n)_n]$. 

Let $B$ be a C*-subalgebra of $A$. Then we will identify $B$ with its image in $A^\infty$ by the embedding
$$a\in B\longmapsto (a,a,\cdots)\in A^\infty.$$
Similarly, we will identify $B^\infty\cap B'$ with the corresponding C*-subalgebra of $A^\infty \cap B'$. Let us denote by $A_\infty$ and $\Ann(B, A^\infty)$ the C*-algebras
$$A_\infty:=A^\infty\cup A', \quad \Ann(B, A^\infty):=\{a\in A^\infty\cap B': ab=0, \,\,\,\forall b\in B\}.$$
Then $\Ann(B, A^\infty)$ is a closed two-sided ideal of $A^\infty\cap B'$. Denote the corresponding quotient by $\F(B, A)$ and the quotient map by $\pi$. We call $\F(A):=\F(A,A)$ the {\it central sequence algebra} of $A$. Note that if $B$ is $\sigma$-unital then $\F(B, A)$ is unital.
Let $\alpha\colon G\to \Aut(A)$ be an action of a finite group $G$ on the C*-algebra $A$ and let $B$ be a $G$-invariant subalgebra of $A$. Then $\alpha$ induces natural actions on $A^\infty$, $A_\infty$, $\F(A)$, and $\F(B,A)$. For simplicity we will denote all these actions again by $\alpha$.


\begin{lemma}\label{lem: aunit}
Let $A$ be a C*-algebra and let $\alpha\colon G\to \Aut(A)$ be an action of a finite group $G$ on $A$. Let $J$ be a $G$-invariant ideal. Then $J$ has a quasi-central approximate unit consisting of $G$-invariant positive contractions. 
\end{lemma}
\begin{proof}
By \cite[Theorem 3.12.14]{Pedersen Book} it is enough to show that $J$ has a $G$-invariant approximate unit. Let $(u_i)_{i\in I}$ be an approximate unit for $J$. For each $i\in I$ put $v_i=\frac{1}{|G|}\sum_{g\in G}\alpha_g(u_i)$. Then it is clear that $\|v_i\|\le 1$ for all $i\in I$, $v_i\le v_j$ for all $i,j\in I$ with $i\le j$, and $\alpha_g(v_i)=v_i$ for all $i\in I$ and $g\in G$. Let $a\in J$ and let $\epsilon>0$. Then there exists $i_0$ such that 
$$\|u_i\alpha_{g^{-1}}(a)-\alpha_{g^{-1}}(a)\|<\epsilon,$$
for all $i\ge i_0$ and $g\in G$. Now using that *-automorphisms are norm preserving we get
$$\|\alpha_g(u_i)a-a\|=\|\alpha_g(u_i\alpha_{g^{-1}}(a)-\alpha_{g^{-1}}(a))\|=\|u_i\alpha_{g^{-1}}(a)-\alpha_{g^{-1}}(a)\|<\epsilon,$$ 
for all $i\ge i_0$ and $g\in G$. Hence, by the triangle inequality
$$\|v_i a-a\|\le\frac{1}{|G|} \sum_{g\in G}\|\alpha_g(u_i)a-a\|<\frac{1}{|G|}(|G|\epsilon)=\epsilon,$$
for all $i\ge i_0$.
This shows that $(v_i)_{i\in I}$ is an approximate unit for $J$.
\end{proof}

\begin{lemma}\label{lem: c}
Let $A$ be a C*-algebra, let $\alpha\colon G\to A$ be an action of a finite group $G$ on $A$, and let $B$ be a $\sigma$-unital C*-subalgebra of $A$. Let $F_1$ be a finite subset of $A^\infty\cap B'$, let $F_2$ be a finite subset of $\Ann(B, A^\infty)$, and let $F_3$ be a finite subset of
$$\{a\in A^\infty: ab=b\,\,\,\text{for all}\,\,\, b\in B\}.$$
Then there exists a positive contraction $d$ in the fixed point algebra $(B^\infty\cap B')^G$ (considered as a subalgebra of $(A^\infty\cap B')^G$) such that 
\begin{align*}
&db=b, \qquad\forall b\in B,\\
&da=ad, \qquad\forall a\in F_1,\\
&da=0, \qquad\forall a\in F_2,\\
&da=a, \qquad\forall a\in F_3.
\end{align*}
\end{lemma}
\begin{proof}
Let $(u_n)_{n\in \N}\subset B^G$ be a countable approximate unit for $B$ (this approximate unit exists by the previous lemma). Let $F_1$, $F_2$, and $F_3$ be as in the statement of the lemma. Without loss of generality we may assume that their cardinalities is the same and that their elements are positive. Write
\begin{align*}
&F_1=\{a^{(i)}=[(a_n^{(i)})_{n}]: 1\le i\le m\},\\
&F_2=\{b^{(i)}=[(b_n^{(i)})_{n}]: 1\le i\le m\},\\
&F_3=\{c^{(i)}=[(c_n^{(i)})_{n}]: 1\le i\le m\}.
\end{align*}
Then for each $j\in \N$ there exists $k_j\in \N$ such that 
\begin{align}\label{eq: d}
\|u_ja_n^{(i)}-a_n^{(i)}u_j\|<\frac 1 j,\qquad \|u_jb_n^{(i)}\|<\frac 1 j, \qquad \|u_jc_n^{(i)}-c_n^{(i)}\|<\frac 1 j,
\end{align}
for $n\ge k_j$ and $1\le i\le m$. Moreover, we may choose the sequence $(k_j)_{j\in \N}$ satisfying $k_j<k_{j+1}$ for all $j\in \N$. Define $d=[(d_n)_{n}]\in B^\infty\subseteq A^\infty$ by
\begin{align*}
d_n=
\begin{cases}
u_1 &\text{ for } 1\le n<k_2\\
u_j &\text{ for } k_{j+1}\le n<k_{j+2}. 
\end{cases}
\end{align*}
Then $d\in B^\infty\cap B'$ and $db=b$ for all $b\in B$, since $(u_n)_{n\in \N}$ is an approximate unit for $B$.  In addition, by the inequalities in \eqref{eq: d} we have
\begin{align*}
da^{(i)}=a^{(i)},\qquad db^{(i)}=0, \qquad dc^{(i)}=c^{(i)},
\end{align*}
for $1\le i\le m$. This concludes the proof of the lemma.
\end{proof}

\subsection{The Cuntz semigroup}
Let us briefly recall the definition of the Cuntz semigroup of a C*-algebra. Let $A$ be a C*-algebra and let $a, b\in A$ be positive elements. We say that $a$ is \emph{Cuntz subequivalent} to $b$, and denote this by $a\precsim b$, if there exists a sequence $(d_n)_{n\in \N}\subset A$ such that $d_n^*bd_n\to a$ as $n\to \infty$. We say that $a$ is \emph{Cuntz equivalent} to $b$, and denote this by $a\sim b$, if $a\precsim b$ and $b\precsim a$. It is easy to see that $\precsim$ is a preorder relation on $A_+$ and that $\sim$ is an equivalence relation. We denote the Cuntz equivalence class of an element $a\in A_+$ by $[a]$. The \emph{Cuntz semigroup} of $A$, denoted by $\Cu(A)$, is the set of Cuntz equivalence classes of positive elements of $A\otimes\K$, where $\K$ denotes the algebra of compact operators on a separable Hilbert space. The addition on $\Cu(A)$ is given by $[a]+[b]:=[a'+b']$, where $a', b'\in (A\otimes \K)_+$ are such that $a\sim a'$, $b\sim b'$, and $a'b'=0$. The preorder relation $\precsim$  induces an order relation on $\Cu(A)$; that is, $[a]\le [b]$ if $a\precsim b$. We say that $[a]$ is \emph{compactly contained} in $[b]$, and denote this by $[a]\ll [b]$, if whenever $[b]\le \sup_n[b_n]$, for some increasing sequence $([b_n])_{n\in \N}$, one has $[a]\le [b_k]$ for some $k$. A sequence $([a_n])_{n\in \N}$ is said to be rapidly increasing if $[a_n]\ll [a_{n+1}]$ for all $n$.  It was shown in \cite{Coward-Elliott-Ivanescu} that $\Cu(A)$ is closed under taking suprema of increasing sequences and that every element of $\Cu(A)$ is the supremum of a rapidly increasing sequence. In particular, we have
$$[(a-\epsilon)_+]\ll[a], \quad [a]=\sup_{\epsilon>0}[(a-\epsilon)_+],$$
for all $a\in (A\otimes\K)_+$ and all $\epsilon>0$.
Here $(a-\epsilon)_+$ denotes the element obtained by evaluating---using functional calculus---the continuous function $f(t)=\max(0, t-\epsilon)$, with $t\in (0,\infty)$, at the element $a$. 

The following lemma states two well known result regarding Cuntz subequivalence. The first statement was proved in \cite[Lemma 2.2]{Kirchberg-Rordam} in the case $\delta=0$ and in \cite[Lemma 1]{Robert-Santiago} for an arbitrary $\delta$. The second statement was shown in \cite[Lemma 2.3]{Kirchberg-Rordam}.

\begin{lemma}\label{lem: Cuntz subequivalence}
Let $A$ be a C*-algebra and let $a, b\in A_+$. The following statements hold:
\begin{itemize}
\item[(i)] If $\|a-b\|<\epsilon$, then $(a-(\epsilon+\delta))_+\precsim (b-\delta)_+$ for all $\delta>0$;
\item[(ii)] If $a\precsim b$ then for every $\epsilon>0$ there exists $\delta>0$ such that $(a-\epsilon)_+\precsim (b-\delta)_+$.
\end{itemize}
\end{lemma}

We say that $\Cu(A)$ is \emph{almost unperforated} if for any $[a], [b]\in \Cu(A)$ satisfying $(n+1)[a]\le n[b]$, for some $n\in \N$, one has $[a]\le [b]$. Consider the set $\mathrm{L}(A)$ of all functionals $\lambda\colon \Cu(A)\to [0, \infty]$ that are additive, order-preserving, and that preserve suprema of increasing sequences. We say that a simple C*-algebra $A$ has \emph{strict comparison of positive elements} if for $[a], [b]\in \Cu(A)$, $\lambda([a])<\lambda([b])$ for all $\lambda\in \mathrm L(A)$ satisfying $\lambda([b])<\infty$, implies $[a]\le [b]$. The following is \cite[Lemma 6.1]{Tikuisis-Toms}.

\begin{lemma}\label{lem: strict-almost}
Let $A$ be a simple C*-algebra. Then $\Cu(A)$ is almost unperforated if and only if $A$ has strict comparison of positive elements.
\end{lemma}  

\begin{lemma}\label{lem: almost quotient}
Let $A$ be a C*-algebra and let $I$ be a $\sigma$-unital closed two-sided ideal of $A$. Suppose that $\Cu(A)$ is almost unperforated, then $\Cu(A/I)$ is almost unperforated. 
\end{lemma}
\begin{proof}
Let $\pi\colon A\to A/I$ denote the quotient map. Let $n\in \N$ and $x,y\in \Cu(A/I)$ be such that 
$$(n+1)x\le ny.$$ 
Choose $a,b\in (A\otimes \K)_+$ such that $[\pi(a)]=x$ and $[\pi(b)]=y$. Then it follows by \cite[Theorem 1.1]{C-R-S} that,
$$(n+1)[a]+z\le n[b]+z,$$
where $z$ is the largest element of $\Cu(I)$ (this element exists because $I$ is $\sigma$-unital). Since $2z=z$ we have
$$(n+1)([a]+z)\le n([b]+[z]).$$
Using now that $\Cu(A)$ is almost unperforated we get $[a]+z\le [b]+z$. Hence, 
$$x=\Cu(\pi)([a]+z)\le \Cu(\pi)([b]+z)=y.$$
This implies that $\Cu(A/I)$ is almost unperforated.
\end{proof}


\subsection{Strongly self-absorbing C*-algebras}
Recall that a unital nuclear C*-algebra $\mathcal{D}$ is said to be {\it strongly self-absorbing} if there exists a *-isomorphism $\phi\colon \mathcal D\to \mathcal D\otimes \mathcal D$ that is approximately unitarily equivalent to the map $i\colon \mathcal D\to \mathcal D\otimes \mathcal D$ given by $i(a)=a\otimes 1_{\mathcal{D}}$ for all $a\in \mathcal D$. A C*-algebra  $A$ is $\mathcal D$-stable if $A\otimes \mathcal D\cong A$. 

The following result was proved in \cite[Proposition 4.1]{Hishberg-Rordam-Winter} under the assumption that the algebra $\mathcal D$ is K$_1$-injective. This assumption is in fact redundant since every strongly self-absorbing C*-algebra is $\mathrm{K}_1$-injective (\cite[Theorem 3.1 and Remark 3.3]{WinterSSA}).
\begin{lemma}\label{lem: Dstability}
Let $A$ and $\mathcal D$ be separable C*-algebras such that $\mathcal D$ is strongly self-absorbing. Then $A$ is $\mathcal D$-stable if and only if for any $\epsilon>0$ and any finite subsets $F\subset A$ and $G\subset \mathcal D$ there exists a completely positive contractive map $\phi\colon \mathcal D\to A$ such that
\begin{itemize}
\item[(i)] $\|b\phi(1_{\mathcal D})-b\|<\epsilon$,
\item[(ii)] $\|b\phi(d)-\phi(d)b\|<\epsilon$,
\item[(iii)] $\|b(\phi(dd')-\phi(d)\phi(d'))\|<\epsilon$,
\end{itemize}
for all $b\in F$ and $d,d'\in G$.
\end{lemma}

\section{The Rokhlin property}
In this section we introduce a Rokhlin-type property for actions of finite groups on not necessarily unital C*-algebras and study several permanence properties of it. We also study its relation with the Multiplier Rokhlin property (\cite[Definition 2.15]{Phillips-freeness}) and with the Rokhlin property as defined in \cite[Definition 3.1]{Nawata}.

\begin{definition}\label{def: Rokhlin}
Let $A$ be a C*-algebra and let $\alpha\colon G\to \Aut(A)$ be an action of a finite group $G$ on $A$. We say that $\alpha$ has the \emph{Rokhlin property} if for any $\epsilon>0$ and any finite subset $F\subset A$ there exist mutually orthogonal positive contractions $(r_g)_{g\in G}\subset A$ such that
\begin{itemize} 
\item[(i)] $\|\alpha_g(r_h)-r_{gh}\|<\epsilon$, for all $g,h\in G$;
\item[(ii)] $\|r_ga-ar_g\|<\epsilon$, for all $a\in F$ and $g\in G$;
\item[(iii)] $\|(\sum_{g\in G} r_g)a-a\|<\epsilon$, for all $a\in F$.
\end{itemize}
The elements $(r_g)_{g\in G}$ will be called {\it Rokhlin elements} for $\alpha$.
\end{definition}

\begin{proposition}\label{prop: Rokhlin equivalence}
Let $A$ be a C*-algebra and let $\alpha\colon G\to \Aut(A)$ be an action of a finite group $G$ on $A$. Then the following are equivalent:
\begin{itemize}
\item[(i)] $\alpha$ has the  Rokhlin property.
\item[(ii)] For any finite subset $F\subset A$  there exist mutually orthogonal positive contractions $(r_g)_{g\in G}\subset A^\infty\cap F'$ such that 
\begin{itemize}
\item[(a)] $\alpha_g(r_h)=r_{gh}$, for all $g,h\in G$;
\item[(b)] $(\sum_{g\in G} r_g)b=b$, for all $b\in F$.
\end{itemize}
\item[(iii)] For any separable C*-subalgebra $B\subseteq A$ there are orthogonal positive contractions $(r_g)_{g\in G}\subset A^\infty\cap B'$ such that
\begin{itemize}
\item[(a)] $\alpha_g(r_h)=r_{gh}$, for all $g,h\in G$;
\item[(b)] $(\sum_{g\in G} r_g)b=b$, for all $b\in B$.
\end{itemize}
\item[(iv)] For any separable $G$-invariant C*-subalgebra $B\subseteq A$ there are projections $(p_g)_{g\in G}\subset \F(B, A)$ such that
\begin{itemize}
\item[(a)] $\alpha_g(p_h)=p_{gh}$, for all $g,h\in G$;
\item[(b)] $\sum_{g\in G} p_g=1_{\F(B, A)}$.
\end{itemize}
\end{itemize}
\end{proposition}
\begin{proof}
(i)$\Rightarrow$(iii). Let $B$ be a separable C*-subalgebra of $A$. Choose finite subsets $(F_n)_{n\in \N}$ of $A$ such that $F_n\subset F_{n+1}$ for all $n$, and $\overline{\bigcup_{n\in\N}F_n}=B$. For each $n$ choose Rokhlin elements $(r_g^{(n)})_{g\in G}$ such that  
\begin{align*}
\|\alpha_g(r^{(n)}_h)-r^{(n)}_{gh}\|<\frac 1 n,\quad \|r^{(n)}_gb-br^{(n)}_g\|<\frac 1 n, \quad 
\|(\sum_{g\in G} r^{(n)}_g)b-b\|<\frac 1 n,
\end{align*}
for all $g, h\in G$ and all $b\in F_n$. Let $r_g=\left(r_g^{(n)}\right)_n\in A^\infty$. Then $\alpha_{g}(r_h)=r_{gh}$ for all $g,h\in G$, and
\begin{align*}
r_gb=br_g, \quad (\sum_{g\in G}r_g)b=b,
\end{align*}
for all $b\in \bigcup_{n\in \N}F_n$. Using that $\bigcup_{n\in\N}F_n$ is dense in $B$ we get the previous equalities for all $b\in B$. This shows that the elements $r_g$, with $g\in G$, belong to $A^\infty\cap B'$ and that the satisfy (a) and (b) of (iii).

(iii)$\Rightarrow$(ii). This is clear by taking $B$ to be the C*-subalgebra generated by the finite set $F$.

(ii)$\Rightarrow$(i). Let $F$ be a finite subset of $A$. Let $(r_g)_{g\in G}\subset A^\infty\cap F'$ be elements satisfying conditions (a) and (b) of (ii). By \cite[Lemma 10.1.12]{Loring} these elements can be lifted to mutually orthogonal positive contractions $((r_{g,n})_n)_{g\in G}\in \ell_\infty(\N, A)$. Let $\epsilon>0$. Then using the definition of $A^\infty$ and that the elements $(r_g)_{g\in G}$ satisfy conditions (a) and (b) of (ii) we get
\begin{align*}
&\|\alpha_g(r_{h,n})-r_{gh,n}\|<\epsilon,\qquad \forall g,h\in G,\\
&\|r_{g,n}b-br_{g,n}\|<\epsilon,\qquad \forall g\in G, \forall b\in F,\\
&\|(\sum_{g\in G}r_{g,n})b-b\|<\epsilon,\qquad \forall b\in F, 
\end{align*}
for some $n\in \N$.
This shows that $\alpha$ has the  Rokhlin property.

(iii)$\Rightarrow$(iv). Let $B$ be a separable $G$-invariant C*-subalgebra of $A$. By (iii) there exist positive orthogonal contractions $(r_g)\subset A^\infty\cap  B'$ satisfying (a) and (b) of (iii). Let $\pi\colon A^\infty\cap B'\to \F(A,B)$ be the quotient map. Set $p_g=\pi(r_g)$. Then by (b) of (iii) we have $\sum_{g\in G}p_g=1_{\F_{B, A}}$. Also, by (a) of (iii) we get $\alpha_g(p_h)=p_{g,h}$ for all $g,h\in G$. Since the elements $(r_g)_{g\in G}$ are orthogonal, the elements $(p_g)_{g\in G}$ are orthogonal. It follows that $p_h^2=p_h\sum_{g\in G}p_g=p_h$ for all $h\in G$. Hence, $(p_g)_{g\in G}$ are projections.

(iv)$\Rightarrow$(ii). Let $F$ be a finite subset of $A$ and let $B$ be the C*-subalgebra of $A$ generated by the finite set $\bigcup_{g\in G}\alpha_g(F)$. Then $B$ is $G$-invariant and separable. Let $(p_g)_{g\in G}\subset \F(B, A)$ be projections satisfying conditions (a) and (b) of (iv). Then by \cite[Lemma 10.1.12]{Loring} there exit mutually orthogonal positive contractions $(f_g)_{g\in G}\subset A^\infty\cap B'$ such that $\pi(f_g)=p_g$. Note that these elements satisfy
$$\alpha_g(f_h)b=f_{gh}b, \quad (\sum_{g\in G}f_g)b=b,$$
for all $g,h\in G$ and all $a\in B$. By Lemma \ref{lem: c} applied to the finite sets $F_1=\{f_g: g\in G\}$, $F_2=\{\alpha_g(f_h)-f_{gh}: g,h\in G\}$, and $F_3=\{\sum_{g\in G}f_g\}$, there exists a positive contraction $d\in (A^\infty\cap B')^G$ such that 
\begin{align*}
db=b, \qquad df_g=f_gd, \qquad \alpha_g(f_h)d=f_{gh}d, \qquad (\sum_{g\in G}f_g)d=d,
\end{align*}
for all $b\in B$, and $g,h\in G$. For each $g\in G$ set $r_g=df_gd$. Then $(r_g)_{g\in G}\subset A^\infty\cap B'$. Also, by the previous equations and that $d\in (A^\infty\cap B')^G$ we get 
\begin{align}\label{eq: rg}
\begin{aligned}
&\sum_{g\in G}r_g=d^2,\\
&\alpha_g(r_h)=r_{gh} \qquad \forall g,h\in G,\\
&r_gr_h=0 \qquad\forall g,h\in G, g\neq h.
\end{aligned}
\end{align}
Therefore, since $(r_g)_{g\in G}\subset A^\infty\cap F'$ and $d^2b=b$ for all $b\in B$, it follows that the elements $(r_g)_{g\in G}$ satisfy conditions (a) and (b) of (ii). 
\end{proof}

\begin{corollary}
Let $A$ be a unital C*-algebra and let $\alpha\colon G\to \Aut(A)$ be an action of a finite group $G$ on $A$. Then $\alpha$ has the standard Rokhlin as in \cite[Definition 1.1]{Phillips-TracialR} (without the assumption of separability) if and only if it has the Rokhlin property as defined in Definition \ref{def: Rokhlin}.
\end{corollary}
\begin{proof}
The forward implication is clear. Let us show the opposite implication. Let $A$ and $\alpha$ be as in the statement of the corollary.  Let $F$ be a finite subset of $A$ and let $B$ be the unital C*-subalgebra of $A$ generated by $\bigcup_{g\in G} \alpha_g(F)$. It is clear that $B$ is separable and $G$-invariant. Also, $\mathrm F(B, A)=A^{\infty}\cap B'$ since $B$ is unital. By (iv) of Proposition \ref{prop: Rokhlin equivalence}  there exists a partition of unity $(p_g)_{g\in G}\subset A^{\infty}\cap B'$ consisting of projections such that
$$\alpha_g(p_h)=p_{gh},$$
for all $g,h\in G$. Let $\phi\colon \C[G]\to A^{\infty}\cap B'$ be the *-homomorphism defined by $\phi(\sum_{g\in G} z_gg)=\sum_{g\in G}z_gp_g$, where $z_g\in \C$ for all $g\in G$. Using that the C*-algebra $\C[G]\cong \C^{|G|}$ is semiprojective by \cite[Lemma 14.1.5 and Theorem 14.2.1]{Loring}, we can lift $\phi$ to a unital map from $\C[G]$ to $\ell_\infty(\N, A)$. In other words, there exists a partition of unity $(q_g)_{g\in G}\subset \ell_\infty(\N, A)$, with $q_g=(q_g^{(n)})_{n\in \N}$,  consisting of projections such that 
\begin{align*}
&\sum_{g\in G}q_g^{(n)}=1, \quad \forall n\in \N,\\
&\lim_{n\to\infty}(\alpha_{g}(q^{(n)}_h)-q^{(n)}_{g,h})=0, \quad \forall g,h\in G,\\
&\lim_{n\to \infty}(q_g^{(n)}a-aq_g^{(n)})=0, \quad \forall a\in F. 
\end{align*}
This implies that $\alpha$ has the Rokhlin property as in \cite[Definition 1.1]{Phillips-TracialR}.
\end{proof}

Let $A$ be a $\sigma$-unital C*-algebra and let $\alpha\colon G\to \Aut(A)$ be an action of a finite group $G$ on $A$. We say that $\alpha$ has the \emph{Rokhlin property in the sense of \cite[Definition 3.1]{Nawata}} if there exist projections $(p_g)_{g\in G}\subset \F(A)$ such that
$$\sum_{g\in G}p_g=1_{\F(A)}, \qquad \alpha_g(p_h)=p_{g,h}\qquad \forall g,h\in G.$$

\begin{corollary}
Let $A$ be a $\sigma$-unital C*-algebra and let $\alpha\colon G\to \Aut(A)$ be an action of a finite group $G$ on $A$. Suppose that $\alpha$ has the Rokhlin property in the sense of \cite[Definition 3.1]{Nawata}, then $\alpha$ has the  Rokhlin property. Moreover, if $A$ is separable then $\alpha$ has the  Rokhlin property in the sense of \cite[Definition 3.1]{Nawata} if and only if it has the Rokhlin property.
\end{corollary}
\begin{proof}
This is a consequence of the equivalence of (i) and (iii) of Proposition \ref{prop: Rokhlin equivalence}.
\end{proof}

Let $A$ be a C*-algebra and let $\alpha\colon G\to \Aut(A)$ be an action of a finite group $G$ on $A$. Let $\M(\alpha)$ denote the extension of $\alpha$ to the multiplier algebra $\M(A)$ of $A$. Recall that $\alpha$ has the \emph{multiplier Rokhlin property} (\cite[Definition 2.15]{Phillips-freeness}) if for every finite set $F\subset A$ and every $\epsilon>0$, there are mutually orthogonal projections $(p_g)_{g\in G}\subset \M(A)$ such that:
\begin{itemize}
\item[(i)] $\|\M(\alpha)_g(p_h)-p_{gh}\|<\epsilon$ for all $g,h\in G$;
\item[(ii)] $\|p_g a-ap_g\|<\epsilon$ for all $g\in G$ and all $a\in F$;
\item[(iii)] $\sum_{g\in G} p_g=1_{\M(A)}$.
\end{itemize}

\begin{corollary}
Let $A$ be a C*-algebra and let $\alpha\colon G\to \Aut(A)$ be an action of a finite group $G$ on $A$. If $\alpha$ has the multiplier Rokhlin property then $\alpha$ has the Rokhlin property.
\end{corollary}
\begin{proof}
Let $A$ and $\alpha$ be as in the statement of the corollary. Let us show that $(A, \alpha)$ satisfies (iv) of Proposition \ref{prop: Rokhlin equivalence}. Let $B$ be a separable $G$-invariant C*-subalgebra of $A$. Choose finite subsets $(F_n)_{n\in\N}$ of $B$ consisting of contractions such that $\bigcup_{n\in \N} F_n$ is dense in the unit ball of $B$. For each $n\in \N$ let $(p^{(n)}_g)_{g\in G}\subset \M(A)$ be mutually orthogonal projections satisfying properties (i), (ii), and (iii) given above for $F_n$ and $\epsilon=\frac{1}{n}$. Select a $G$-invariant approximate unit $(u_\lambda)_{\lambda\in \Lambda}$ for $A$ that is quasi-central for $\M(A)$ (this approximate unit exists by Lemma \ref{lem: aunit}). Let $n\in \N$ be fixed. Then for each  $a\in F_n$ we have
\begin{align*}
&\lim_\lambda\|(u_\lambda p^{(n)}_g u_\lambda)(u_\lambda p^{(n)}_h u_\lambda)a\|=\lim_\lambda\|p^{(n)}_g p^{(n)}_h au^4_\lambda\|=0<\frac 1 n,\quad \forall g,h\in G,\,\,\, g\neq h\\
&\lim_\lambda\|\alpha_g(u_\lambda p^{(n)}_h u_\lambda)a-u_\lambda p^{(n)}_{gh} u_\lambda a\|=\lim_\lambda\|(\M(\alpha)_g(p^{(n)}_h)- p^{(n)}_{gh})a u^2_\lambda\|<\frac 1 n,\quad \forall g,h\in G,\\
&\lim_\lambda\|(u_\lambda p^{(n)}_g u_\lambda) a-a(u_\lambda p^{(n)}_g u_\lambda)\|=\lim_\lambda\|(p^{(n)}_g  a-a p^{(n)}_g) u^2_\lambda\|<\frac 1 n,\quad \forall g\in G,\\
& \lim_\lambda\|(\sum_{g\in G}u_\lambda p^{(n)}_g u_\lambda)a-a\|=\|(\sum_{g\in G} p^{(n)}_g)a-a\|=0<\frac 1 n,
\end{align*}
Therefore, there exists a positive contraction $u\in A$ such that for $r^{(n)}_g:=up^{(n)}_gu\in A$ the following hold:
\begin{align*}
&\|r^{(n)}_gr^{(n)}_h a\|<\frac 1 n,\quad \forall g,h\in G,\,\,\, g\neq h, \quad \forall a\in F_n\\ 
&\|(\alpha_g(r^{(n)}_h)-r^{(n)}_{gh})a\|<\frac 1 n,\quad\forall g,h\in G,\quad \forall a\in F_n\\
&\|r^{(n)}_ga-ar^{(n)}_g\|<\frac 1 n,\quad\forall g\in G,\quad \forall a\in F_n\\
&\|(\sum_{g\in G}r^{(n)}_g)a-a\|<\frac 1 n,\quad \forall a\in F_n.
\end{align*}
For each $g\in G$ set $r_g:=(r_g^{(n)})_n\in \ell_\infty(\N,A)$. Then by the inequalities above, for each $a\in \bigcup_{n\in \N} F_n$ we have $r_gr_h a=0$ for all $g,h\in G$ with $g\neq h$, $(\alpha_g(r_h)-r_{gh})a=0$ for all $g,h\in G$, $ar_g=r_g a$ for all $g\in G$, and $(\sum_{g\in G}r_g)a=a$. Since $\bigcup_{n\in \N} F_n$ is dense in the unit ball of $B$, the same equalities hold for every $a\in B$. Let $\pi\colon A^\infty\cap B'\to F(B, A)$ be quotient map. For each $g\in G$ put $p_g=\pi([r_g])$. Then $(p_g)_{g\in G}$ are mutually orthogonal projections that satisfy conditions (a) and (b) of (iv) of Proposition \ref{prop: Rokhlin equivalence}. Therefore, by the same proposition $\alpha$ has the Rokhlin property.
\end{proof}

The following is an example of a $\Z_2$-action with the Rokhlin property that does not have the multiplier Rokhlin property:

\begin{example}
Consider the space $X=\{\frac 1 n: n\in \N\}$ with the topology induced from $\R$. Let $\phi\colon X\to X$ be the map  defined by $\phi(1/(2n-1))=1/(2n)$, and $\phi(1/(2n))=1/(2n-1)$,
for all $n\in \N$. It is clear that $\phi$ is bijective, continuous, proper, its inverse is continuous, and that $\phi^2=\id_X$. This implies that the map $\alpha\colon \mathrm C_0(X)\to \mathrm C_0(X)$ defined by $\alpha(f)(x)=f(\phi^{-1}(x))$ is a *-automorphism satisfying $\alpha^2=\id_{\mathrm C_0(X)}$. Consider the action of $\Z_2$ on $\mathrm C_0(X)$ defined by this automorphism. Let us see that this action has the Rokhlin property. For each $n\in \N$, let $r_n^{(1)}, r_n^{(2)}\in \mathrm C_0(X)$ denote the characteristic functions of the open subsets $\{1/(2k-1) : 1\le k\le n\}$ and $\{1/(2k) : 1\le k\le n\}$, respectively (note that these functions are continuous since the given sets are open). Then $r_n^{(1)}$ and $r_n^{(2)}$ are orthogonal, $\alpha(r_n^{(1)})=r_n^{(2)}$, and $(r_n^{(1)}+r_n^{(2)})_{n\in \N}$ is an approximate unit for $\mathrm C_0(X)$. Since $\mathrm C_0(X)$ is commutative it follows that the given $\Z_2$-action has the Rokhlin property. It is clear that this action does not have the multiplier Rokhlin property since $X$ can not be written as the union of two disjoint open sets $X_1$ and $X_2$ satisfying $\phi(X_1)=X_2$. (If $X=X_1\cup X_2$ where $X_1$ and $X_2$ are disjoint and open, then either $X_1$ or $X_2$ contains a subset of $X$ of the form $\{1/n: n>k\}$. In particular, it follows that $\phi(X_1)\neq X_2$.) 
\end{example}

Let $A$ be a C*-algebra and let $\alpha\colon G\to \Aut(A)$ be an action of a finite group $G$ on $A$. Recall that a map $\omega\colon G\to \mathrm U(\M(A))$, where $\mathrm U(\M(A))$ denotes the unitary group of $\M(A)$, is said to be an \emph{$\alpha$-cocycle} if it satisfies $\omega_{gh}=\omega_g\alpha_g(\omega_h)$ for all $g, h\in G$. Let $\beta\colon G\to \Aut(A)$ be an action of $G$ on $A$. The actions $\alpha$ and $\beta$ are said to be \emph{exterior equivalent} if there exists an $\alpha$-cocycle $\omega$ such that $\beta_g(a)=\omega_g\alpha_g(a)\omega_g^*$ for all $g\in G$ and all $a\in A$. 

\begin{theorem}
Let $A$ and $B$ be C*-algebras and let $\alpha\colon G\to \Aut(A)$ be an action of a finite group $G$ on $A$. The following statements hold:
\begin{itemize}
\item[(i)] If $\beta\colon G\to \Aut(B)$ is an action with the  Rokhlin property, then the tensor action 
$$\alpha\otimes_{\min} \beta\colon G\to \Aut(A\otimes_{\min} B)$$ 
has the  Rokhlin property. Here $\otimes_{\min}$ denotes the minimal tensor product;

\item[(ii)] If $\alpha$ has the Rokhlin property and $B$ is a $G$-invariant hereditary C*-subalgebra of $A$ then the restriction of $\alpha$ to $B$ has the Rokhlin property;

\item[(iii)] If $\alpha$ has the Rokhlin property and $I$ is a $G$-invariant closed two-sided ideal of $A$ then the action induced by $\alpha$ on $A/I$ has the Rokhlin property;

\item[(iv)] If for any finite subset $F\subset A$ and any $\epsilon>0$ there exist a C*-algebra $B$ with an action $\beta\colon G\to \Aut(B)$ with the  Rokhlin property, and an equivariant *-homomorphism $\phi\colon B\to A$ such that $\dist(a, \phi(B))<\epsilon$ for all $a\in F$. Then $\alpha$ has the  Rokhlin property;

\item[(v)]  If $(A, \alpha)$ is a direct limit---with equivariant connecting *-homomorphisms---of a direct system $(A_\lambda, \alpha_{\lambda})_{\lambda\in \Lambda}$, and every action $\alpha_{\lambda}\colon G\to \mathrm{Aut}(A_\lambda)$ has the  Rokhlin property, then $\alpha$ has the  Rokhlin property;

\item[(vi)] If $\alpha$ has the Rokhlin property and $H$ is a subgroup of $G$ then the restriction of $\alpha$ to $H$ has the Rokhlin property.

\item[(vii)] If $\beta\colon G\to \Aut(A)$ is an action with the Rokhlin property and $\alpha$ and $\beta$ are exterior equivalent then $\alpha$ has the Rokhlin property.
\end{itemize}
\end{theorem}
\begin{proof}
(i) Given that $G$ is finite the tensor action $\alpha\otimes_{\min}\beta$ is well defined. Let us show that $\alpha\otimes \beta$ has the Rokhlin property. Since $A\otimes_{\min} B$ is the closure of the span of simple tensors, it is enough to show that for any $\epsilon>0$ and any finite set $F\subset A\otimes_{\min} B$ consisting of simple tensors, there exist orthogonal positive contractions $(f_g)_{g\in G}$ in $A\otimes_{\min} B$ satisfying (i), (ii), and (iii) of Definition \ref{def: Rokhlin}. Let $\epsilon>0$ and 
$$F=\{a_k\otimes b_k: a_k\in A,\, b_k\in B,\, 1\le k\le n\}.$$ 
Choose an approximate unit $(u_i)_{i\in I}$ for $A$ consisting of $G$ invariant elements (this approximate unit exists by Lemma \ref{lem: aunit}). Then there exists $i\in I$ such that 
\begin{align}\label{eq: ua}
\|u_ia_k-a_ku_i\|<\frac \epsilon 2, \qquad \|u_ia_k-a_k\|<\frac \epsilon 2,
\end{align}
for $1\le k\le n$. Choose Rokhlin elements $(r_g)_{g\in G}\subset B$ for $\beta$ corresponding to the finite set $\{b_1, b_2, \cdots, b_n\}$ and to the number $\frac \epsilon 2$. For each $g\in G$ put $f_g=u_i\otimes r_g$. Then the elements $(f_g)_{g\in G}$ are mutually orthogonal since the elements $(r_g)_{g\in G}$ are orthogonal. Now by the triangle inequality, the choice of $(r_g)_{g\in G}$, and the inequalities in \eqref{eq: ua} we get
\begin{align*}
&\|(\alpha\otimes\beta)_g(f_h)-f_{gh}\|\le \|u_i\otimes(\beta_g(r_h)-r_{gh})\|\le \frac \epsilon 2<\epsilon,\\
&\|f_g(a_k\otimes b_k)-(a_k\otimes b_k)f_g\|\le\|(u_ia_k-a_ku_i)\otimes r_gb_k\|+\|a_ku_i\otimes(r_gb_k-b_kr_g)\|<\frac \epsilon 2+\frac \epsilon 2=\epsilon,\\
&\|(\sum_{g\in G}f_g)(a_k\otimes b_k)-(a_k\otimes b_k)\|\le \|(u_ia_k-a_k)\otimes \sum_{g\in G}r_gb_k\|+\|a_k\otimes (\sum_{g\in G}r_gb_k-b_k)\|<\frac \epsilon 2+\frac \epsilon 2=\epsilon,
\end{align*}
for $1\le k\le n$. We have shown that $(f_g)_{g\in G}$ are Rokhlin elements for $\beta$ corresponding to the finite set $F$ and the number $\epsilon$. This implies that $\beta$ has the  Rokhlin property.

(ii) Let $A$, $B$, and $\alpha$ be as in the statement of (ii). Let $F$ be a finite subset of $B$. Then by (ii) of Proposition \ref{prop: Rokhlin equivalence} there exist orthogonal positive contractions $(r_g)_{g\in G}\subset A^\infty\cap F'$ satisfying $\alpha_g(r_h)=r_{gh}$ for all $g,h\in G$, and $(\sum_{g\in G}r_g)a=a$ for all $a\in F$. By Lemma \ref{lem: c} applied to the separable C*-subalgebra of $B$ generated by the elements in $F$ and to the finite sets $F_1=\{r_g: g\in G\}$, $F_2=\{0\}$, and $F_3=\{\sum_{g\in G}r_g\}$, there exists a positive contraction $d\in B^\infty\cap B'\subset A^\infty\cap B'$ such that $da=a$ for all $a\in F$, $dr_g=r_gd$ for all $g\in G$, and $(\sum_{g\in G}r_g)d=d$. For each $g\in G$ set $f_g=dr_gd$. Then $f_g\in B^\infty\cap B'$ since $B$ is a hereditary C*-subalgebra of $A$. Also, it is straightforward to check that the elements $(f_g)_{g\in G}$ satisfy (a) and (b) of (ii) of Proposition \ref{prop: Rokhlin equivalence}. Therefore, by the equivalence of (i) and (ii) in the same proposition, the restriction of $\alpha$ to $B$ has the  Rokhlin property. 

(iii) Let $\epsilon>0$ and let $F\subset A/I$ be a finite set. Let $\pi\colon A\to A/I$ denote the quotient map. Choose a finite set $\widetilde F\subset A$ such that $\pi(\widetilde F)=F$. Choose Rokhlin elements $(r_g)_{g\in G}$ for $\widetilde F$ and $\epsilon>0$. Then it is straightforward that $(\phi(r_g))_{g\in G}$ are Rokhlin elements for $F$ and $\epsilon$. 

(iv) Let $A$ and $\alpha$ be as in the statement of the proposition. Let $F=\{a_1, a_2,\cdots a_n\}$ be a finite subset of $A$. Then by assumption there exist a C*-algebra $B$ with an action $\beta\colon G\to \Aut(B)$ with the  Rokhlin property, an equivariant *-homomorphism $\phi\colon B\to A$, and a finite subset $\widetilde F=\{b_1, b_2, \cdots b_n\}\subset B$ such that $\|a_i-\phi(b_i)\|<\frac \epsilon 3$ for $1\le i\le n$. 
Since $\beta$ has the  Rokhlin property, there exist orthogonal positive contractions $(r_g)_{g\in G}\subset B$ such that
$$\|\alpha_g(r_h)-r_{gh}\|<\epsilon, \quad \|r_gb_i-b_ir_g\|<\frac{\epsilon}{3},\quad  \|(\sum_{g\in G}r_g)b_i-b_i\|<\frac{\epsilon}{3},$$
for $g,h\in G$ and $1\le i\le n$. For each $g\in G$ set $f_g:=\phi(r_g)$. Then $(f_g)_{g\in G}$ are orthogonal positive contractions. Using that $\phi$ is equivariant we get 
$$\|\beta_g(f_h)-f_{gh}\|=\|\phi(\alpha_g(r_h)-r_{gh})\|\le \|\alpha_g(r_h)-r_{gh}\|<\epsilon.$$
Also, using the triangle inequality we get
\begin{align*}
&\|f_ga_i-a_if_g\|\le \|\phi(r_gb_i-b_ir_g)\|+\|f_g(a_i-\phi(b_i))\|+\|(a_i-\phi(b_i))f_g\|<\frac{\epsilon}{3}+\frac{\epsilon}{3}+\frac{\epsilon}{3}=\epsilon,\\
&\|(\sum_{g\in G}f_g)a_i-a_i\|\le\|\phi((\sum_{g\in G}r_g)b_i-b_i)\|+\|(\sum_{g\in G}f_g)(a_i-\phi(b_i))\|+\|a_i-\phi(b_i)\|<\frac{\epsilon}{3}+\frac{\epsilon}{3}+\frac{\epsilon}{3}=\epsilon,
\end{align*}
for $g\in G$ and $1\le i\le n$. Hence, $(f_g)_{g\in G}$ are Rokhlin elements for $\alpha$, $F$ and $\epsilon$. 
It follows that $\beta$ has the  Rokhlin property.

(v) This is a particular case of (iv).

(vi) Let $H$ be a subgroup of $G$ and let $n$ denote the number of elements of $G/H$. Let $F\subset A$ and let $\epsilon>0$. Choose a family $K$ of representatives of $H$ in $G$ and Rokhlin elements $(r_g)_{g\in G}$ for $\alpha$, $F$, and $\epsilon/n$. Then it is easy to check using the triangle inequality that $(f_h)_{h\in H}$, with $f_h=\sum_{g\in K}f_{hg}$ for all $h\in H$, are Rokhlin elements for the restriction of $\alpha$ to $H$, the finite set $F$, and the number $\epsilon>0$.

(vii) Let $\alpha, \beta\colon G\to \Aut(A)$ be actions of $G$ on $A$ that are exterior equivalent and suppose that $\alpha$ has the Rokhlin property. Let $\omega\colon G\to \mathrm U(\M(A))$ be the $\alpha$-cocycle implementing the equivalence. Let $F$ be a finite subset of $A$. Using Lemma \ref{lem: aunit} choose an approximate unit $(u_i)_{i\in I}\subset A$ for $A$ that is quasicentral in $\M(A)$ and that satisfies $\alpha_g(u_i)=u_i$ for all $g\in G$ and $i\in I$. Then for each $n\in \N$ there exists $i_n\in I$ such that 
 \begin{align*}
&\|u_{i_n}a-a\|<\frac 1 n,\quad \forall a\in F,\\
&\|u_{i_n}w_g-w_gu_{i_n}\|<\frac 1 n, \quad \forall g\in G.
\end{align*}
For each $n\in \N$ choose Rokhlin elements $(r_g^{(n)})_{g\in G}\subset A$ for the the finite set $F_n=F\cup\{u_{i_n}w_g: g\in G\}$ and the number $\frac 1 n$. Consider the elements $u=[(u_n)_n], r_g=[(r_g^{(n)})]\in A^\infty$. Then it follows that
\begin{eqnarray}
&ua=au=a,
&uw_g=w_gu, \label{eq: 1}\\
&\alpha_g(r_h)=r_{gh},
&r_ga=ar_g,\quad r_guw_g=uw_gr_g, \label{eq: 2}\\
&(\sum_{g\in G}r_g)a =a,
&(\sum_{g\in G}r_g)uw_h =uw_h,\label{eq: 3}
\end{eqnarray}
for all $a\in F$ and $g,h\in G$. Note that by the third equality in \eqref{eq: 2} we have
$$u^2r_g=(uw_g)(uw_g)^*r_g=(uw_g)r_g(uw_g)^*=r_g(uw_g)(uw_g)^*=r_gu^2,$$
for all $g\in G$.
This implies that $f(u^2)r_g=r_gf(u^2)$ for every $f\in \mathrm{C}_0((0, \|u\|], \R_+)$, and in particular for square root function. Therefore, $ur_g=r_gu$ for all $g\in G$.
For each $g\in G$ set $f_g:=ur_gu$. Then $f_g=u^2r_g=r_gu^2$. In addition, by the equalities in \eqref{eq: 1}, \eqref{eq: 2}, and \eqref{eq: 3} we get
\begin{align*}
\beta_g(f_h)=w_g\alpha_g(ur_hu)w_g^*=w_gur_{gh}uw_g^*=uw_gr_{gh}(uw_g)^*=r_{gh}(uw_g)(uw_g)^*=r_{gh}u^2=f_{gh},
\end{align*}
for all $g,h\in G$, 
\begin{align*}
f_gf_h=(ur_gu)(ur_hu)=u^2r_gr_hu^2=0,
\end{align*}
for all $g,h\in G$ with $g\neq h$,
\begin{align*}
f_ga=ur_gua=ur_gau=aur_gu=af_g,
\end{align*}
for all $g\in G$ and $a\in F$, and
\begin{align*}
(\sum_{g\in G}f_g)a=\sum_{g\in G}ur_gua=u(\sum_{g\in G}r_g)a=ua=a, 
\end{align*}
for all $a\in F$. This shows that $(f_g)_{g\in G}$ are elements of $A^\infty\cap F'$ that satisfy conditions (a) and (b) of (ii) of Proposition \ref{prop: Rokhlin equivalence} for the action $\beta$ and the finite set $F$. Therefore, it follows by the same proposition that $\beta$ has the Rokhlin property.    
\end{proof}

Recall that an action $\alpha\colon G\to \Aut(A)$ is said to be \emph{pointwise outer} if, for $g\in G \setminus\{e\}$, the automorphism $\alpha_g$ is outer, that is, not of the form $a\mapsto \Ad(u)(a)=uau^*$ for some unitary $u$ in the multiplier algebra $\M(A)$ of $A$. 

\begin{proposition}\label{prop: outer}
Let $A$ be C*-algebra and let $\alpha\colon G\to \Aut(A)$ be an action of a finite group $G$ on $A$ with the Rokhlin property. Then $\alpha$ is pointwise outer.
\end{proposition}
\begin{proof}
Suppose that there exist $g_0\in G\setminus \{e\}$ and a unitary $u$ in the multiplier algebra $\M(A)$ of $A$ such that $\alpha_{g_0}(a)=uau^*$ for all $a\in A$. Let $b$ be a nonzero element of the fixed-point algebra $A^\alpha$. By (ii) of Proposition \ref{prop: Rokhlin equivalence} applied to the finite set $F=\{b, bu^*\}$, there exist orthogonal positive contractions $(r_g)_{g\in G}\subset A^\infty$ such that 
$$\alpha_g(r_h)=r_{g,h}, \quad r_g(bu^*)=(bu^*)r_g,\quad r_gb=br_g\quad  (\sum_{g\in G}r_g)b=b,$$
for all $g,h\in G$. Note that $r_eb\neq 0$ since $\sum_{g\in G}\alpha_g(r_eb)=(\sum_{g\in G}r_g)b=b\neq 0$. Also, we have $(r_{g_0}b)(br_e)=(br_e)(r_{g_0}b)=0$ and 
$$ur_ebu^*=u(bu^*)r_e=\alpha_{g_0}(b)r_e=br_e.$$
Hence,
$$0=\|\alpha_{g_0}(r_eb)-ur_ebu^*\|=\|r_{g_0}b-br_e\|=\max(\|r_{g_0}b\|, \|br_e\|)\neq 0,$$
which is a contradiction. Therefore $\alpha_g$ is outer for all $g\in G\setminus\{e\}$.
\end{proof}

Let $A$ be a C*-algebra and let $\alpha\colon G\to \Aut(A)$ be an action of a finite group $G$ on $A$. Let us denote by $\mathrm{Ideal}(A)$ the set of closed two-sided ideals of $A$, by $\mathrm{Ideal}(A)^G\subset \mathrm{Ideal}(A)$ the set of $G$-invariant ideals, and by $\mathrm{Ideal}(a)$ the closed two-sided ideal of $A$ generated by $a\in A$. The following result is essentially \cite[Theorem 1.25]{Sierakowski}.

\begin{proposition}
Let $A$ be C*-algebra and let $\alpha\colon G\to \Aut(A)$ be an action of a finite group $G$ on $A$ with the Rokhlin property. Then the map
\begin{align*}
\mathrm{Ideal}(A\rtimes_\alpha G)&\longrightarrow \mathrm{Ideal}(A)^G,\\
I&\longmapsto I\cap A,
\end{align*}
is a bijection. In particular, every ideal of $J\subset A\rtimes_\alpha G$ has the form $I\rtimes_{\alpha_{(\cdot)|_I}} G$ for some $G$-invariant ideal $I$ of $A$.
\end{proposition}
\begin{proof}
The map above is clearly onto since $(I\rtimes_\alpha G)\cap A=I$ for all $I\in \mathrm{Ideal}(A)^G$. Let us show that it is one-to-one. Let $(u_g)_{g\in G}\subset \M(A\rtimes_\alpha G)$ be the canonical unitaries implementing the action $\alpha$. By \cite[Theorem 1.10(ii)]{Sierakowski} it is enough to show that for $x\in A\rtimes_\alpha G$, with $x=\sum_{g\in G}a_gu_g$, one has $a_e\in \mathrm{Ideal}(x)$.

Let $x=\sum_{g\in G}a_gu_g$ be an element of $A\rtimes_\alpha G$, with $a_g\in A$ for all $g\in G$. Choose Rokhlin elements $(r_g^{(n)})_{g\in G}$, with $n\in \N$, associated to the finite set $F=\{a_g: g\in G\}$ and to the numbers $\frac{1}{n}$, with $n\in \N$. For each $g\in G$ set $r_g=(r_g^{(n)})_{n}$. Then $r_g\in A^\infty\cap F'\subset (A\rtimes_\alpha G)^\infty\cap F'$ for all $g\in G$,
$\alpha_g(r_h)=r_{gh}$ for all $g,h\in G$, and $\sum_{g\in G}r_g a_e=a_e$. It follows that
$$\sum_{h\in G}r_hxr_h= \sum_{g\in G}\sum_{h\in G}r_ha_gu_gr_h=\sum_{g\in G}\sum_{h\in G}\delta_{h, gh}r^2_ha_gu_g=\sum_{h\in G}r_h^2a_e=(\sum_{h\in G}r_h)^2a_e=a_e.$$
Therefore, $\sum_{h\in G}r^{(n)}_hxr^{(n)}_h\to a_e$ as $n\to \infty$. This implies that $a_e\in \mathrm{Ideal}(x)$.
\end{proof}

\begin{corollary}
Let $A$ be a simple C*-algebra and let $\alpha\colon G\to \Aut(A)$ be an action of a finite group $G$ on $A$ with the Rokhlin property. Then $A\rtimes_\alpha G$ is simple.
\end{corollary}

\begin{proposition}\label{prop: crossed-fixed}
Let $A$ be C*-algebra and let $\alpha\colon G\to \Aut(A)$ be an action of a finite group $G$ on $A$ with the Rokhlin property. Then the crossed product C*-algebra $A\rtimes_\alpha G$ and the fixed point algebra $A^G$ are strongly Morita equivalent. 
\end{proposition}
\begin{proof}
This is a consequence of the discussion after \cite[Lemma 3.1]{G-L-P} (see also  \cite[Theorem 5.10]{Phillips-freeness}). 
\end{proof}

\begin{corollary}
Let $A$ be C*-algebra and let $\alpha\colon G\to \Aut(A)$ be an action of a finite group $G$ on $A$ with the Rokhlin property. Then $\Cu(A^G)$ and $\Cu(A\rtimes_\alpha G)$ are naturally isomorphic.
\end{corollary}
\begin{proof}
By Proposition \ref{prop: crossed-fixed} the crossed product C*-algebra $A\rtimes_\alpha G$ and the fixed point algebra $A^G$ are strongly Morita equivalent. Hence, there exists a canonical C*-algebra $B$ containing copies of $A\rtimes_\alpha G$ and $A^G$ as full hereditary C*-subalgebras (\cite[page 288]{Rieffel Moritaeq}). It follows that $\Cu(A\rtimes_\alpha G)\cong \Cu(B)\cong \Cu(A^G)$.  
\end{proof}

\section{Local approximation of $A\rtimes_\alpha G$}
In this section we show that the crossed product C*-algebra by an action of a finite group on a C*-algebra with the Rokhlin property can be locally approximated by matrix algebras over hereditary C*-subalgebras of the given C*-algebra.


Let $(f_{i,j})_{i,j=1}^d$ denote the set of matrix units of $\M_d(\C)$.
The following result is well known. We include the proof for the convenience of the reader.

\begin{lemma}\label{lem: matrix units}
Let $A$ be a C*-algebra, let $\phi\colon \mathrm{C}_0(0,1]\otimes \M_d(\C)\to A$ be a *-homomorphism, and let $b=\phi(t\otimes 1_{\M_d})$ and $c=\phi(t\otimes f_{1,1})$. Then the map $\psi\colon \M_d(cAc)\to \overline{bAb}$ defined by
$$\psi(\sum_{i,j=1}^d(cx_{i,j}c)\otimes f_{i,j}):=\sum_{i,j=1}^d\phi(t\otimes f_{i,1})x_{i,j}\phi(t\otimes f_{1,j}),$$
extends to a *-isomorphism $\widetilde\psi\colon \M_d(\overline{cAc})\to \overline{bAb}$.
\end{lemma}
\begin{proof}
The map $\psi$ is clearly additive. Also, using that $(f_{i,j})_{i,j=1}^d$ is a set of matrix units of $\M_d(\C)$ it is easy to check that $\psi$ is multiplicative and that it preserves the adjoint operation.

Now note that the matrix $A_n=\left(\delta_{i,1}\phi(t^\frac{1}{n}\otimes f_{i,j})\right)_{1\le i,j\le d}$, where $\delta_{i,j}$ denotes Kronecker's delta function, satisfies
$$\|A_n\|=\|A_n^*A_n\|^\frac{1}{2}=\|\sum_{i=1}^d\phi(t^\frac{2}{n}\otimes f_{i,i})\|^\frac{1}{2}\le 1.$$
Hence, we have
\begin{align*}
\|\psi(\sum_{i,j=1}^n(cx_{i,j}c)\otimes f_{i,j})\|&=\|\sum_{i,j=1}^n\phi(t\otimes f_{i,1})x_{i,j}\phi(t\otimes f_{1,j})\|\\
&=\lim_{n\to \infty}\|\sum_{i,j=1}^n\phi(t^\frac{1}{n}\otimes f_{i,1})cx_{i,j}c\phi(t^\frac{1}{n}\otimes f_{1,j})\|\\
&=\lim_{n\to \infty}\|A_n^*\left(cx_{i,j}c\right)_{i,j}A_n\|\\
&\le \|\left(cx_{i,j}c\right)_{i,j}\|\\
&= \|\sum_{i,j=1}^n(cx_{i,j}c)\otimes f_{i,j}\|,
\end{align*} 
and
\begin{align*}
\|\sum_{i,j=1}^n(cx_{i,j}c)\otimes f_{i,j}\|&=\|\left(cx_{i,j}c\right)_{i,j}\|\\
&=\lim_{n\to \infty}\|A_n^t\left(\phi(t\otimes f_{i,1})x_{i,j}\phi(t\otimes f_{1,j})\right)_{i,j}(A_n^t)^*\|\\
&=\lim_{n\to \infty}\|\left(\phi(t^{1+\frac 1 n}\otimes f_{1,1})x_{i,j}\phi(t^{1+\frac 1 n}\otimes f_{1,1})\right)_{i,j}\|\\
&=\lim_{n\to \infty}\|A_n(\sum_{i,j=1}^d\phi(t\otimes f_{i,1})x_{i,j}\phi(t\otimes f_{1,j}))A_n^*\|\\
&\le \|\sum_{i,j=1}^d\phi(t\otimes f_{i,1})x_{i,j}\phi(t\otimes f_{1,j})\|\\
=&\|\psi(\sum_{i,j=1}^n(cx_{i,j}c)\otimes f_{i,j})\|,
\end{align*}
for all $x_{i,j}\in A$. Therefore, $\psi$ is an isometry and in particular it is well defined. Now using that $\M_d(cAc)$ is dense in $\M_d(\overline{cAc})$ we can extend $\psi$ by continuity to an isometric *-homomorphism $\widetilde\psi\colon \M_d(\overline{cAc})\to \overline{bAb}$. Note that
$$\widetilde\psi(c^3\otimes 1_{\M_d})=\psi(c^3\otimes 1_{\M_d})=\sum_{i=1}^d\phi(t^3\otimes f_{i,i})=b^3.$$
Using that $\widetilde\psi$ preserves functional calculus we get $\widetilde\psi(c)=b$. This implies that $\widetilde\psi$ is surjective. Therefore, it is a *-isomorphism.
\end{proof}

\begin{lemma}\label{lem: MvN}
Let $A$ be a C*-algebra and let $0<\epsilon<1$. Let $a,b\in A_+$ and $x\in A$ be such that 
$$\|a-x^*x\|<\epsilon, \quad \|b-xx^*\|<\epsilon.$$
Then there exists $y\in A$ such that
$$(a-9\epsilon^{\frac 1 2})_+=y^*y,\quad yy^*\le (b-\epsilon)_+, \quad \|y-x\|<16\epsilon^\frac 1 4.$$
\end{lemma}
\begin{proof}
Let $a$, $b$, and $x$ be as in the statement of the lemma. Then by \cite[Proposition 1]{Robert-Santiago} there exists $z\in A$ such that 
$$(b-\epsilon)_+=zz^*, \quad z^*z\le x^*x,\quad \|z-x\|<4 \epsilon^{\frac 1 2}.$$
(A straightforward computation shows that \cite[Proposition 1]{Robert-Santiago} holds for $C = 4$.) 
In particular, we have
$$\|a-z^*z\|\le \|a-x^*x\|+\|(x-z)^*x\|+\|z^*(x-z)\|< \epsilon+4\epsilon^{\frac 1 2}+4\epsilon^{\frac 1 2}\le 9\epsilon^{\frac 1 2}.$$
By applying again \cite[Proposition 1]{Robert-Santiago}, this times to the elements $a$ and $z$, we get an element $y\in A$ satisfying
$$(a-9\epsilon^{\frac 1 2})_+=y^*y, \quad yy^*\le zz^*=(b-\epsilon)_+, \quad \|y-z\|<12\epsilon^\frac{1}{4}.$$
A simple computation now shows that $\|y-x\|<16\epsilon^\frac 1 4$.
\end{proof}

In order to prove the main theorem of this section we need the following lifting result. This is a slight refinement of \cite[Theorem 10.2.2]{Loring} for quotients of the form $\ell_\infty(\N, A)/\mathrm{c}_0(\N, A)$.

\begin{proposition}\label{prop: lifting}
Let $A$ be a C*-algebra and let $0<\epsilon<1$. Given a *-homomorphism $\phi\colon \M_d(\CC_0(0,1])\to A^\infty$ and orthogonal positive contractions $h_1, h_2,\cdots, h_d\in \ell_\infty(\N, A)$ such that $[h_i]=\phi(t\otimes f_{i,i})$ for all $i\in \N$, there exists a *-homomorphism $\psi\colon\M_d(\CC_0(0,1])\to \ell_\infty(\N, A)$ that lifts $\phi$ and that satisfies
$$\psi(t\otimes f_{1,1})\in \CC^*(h_1), \quad \psi(t\otimes f_{i,i})\le h_i.$$ 
\end{proposition}
\begin{proof}
Let $\phi\colon \M_d(\CC_0(0,1])\to A^\infty$ and $h_1=(h_{1,n}), h_2=(h_{2,n}),\cdots, h_d=(h_{d,n})\in \ell_\infty(\N, A)$ be as in the statement of the proposition. For each $1\le i\le d$ let $x_i=(x_{i,n})\in \ell_\infty(\N, A)$ be a lifting of $\phi(t\otimes f_{i,1})$. Then
$$\|h_{1,n}^2-x_{i,n}^*x_{i,n}\|\to 0,\quad \|h_{i,n}^2-x_{i,n}x_{i,n}^*\|\to 0,$$
as $n\to \infty$. By passing to a subsequence and relabeling we may assume that 
$$\|h_{1,n}^2-x_{i,n}^*x_{i,n}\|<\frac 1 n,\quad \|h_{i,n}^2-x_{i,n}x_{i,n}^*\|<\frac 1 n,$$
for all $1\le i\le d$ and $n\in \N$. Using these inequalities and Lemma \ref{lem: MvN} we can find elements $y_{i,n}$ such that
$$y_{i,n}^*y_{i,n}=(h_{1,n}-9\epsilon^{-\frac 1 2})_+, \quad y_{i,n}y_{i,n}^*\le h_{i,n}^2, \quad \|y_{i,n}-x_{i,n}\|<16\epsilon^\frac{1}{4}.$$ 
For each $i$ set $y_i=(y_{i,n})_{n}$. Then it follows that $y_i$ is an element in $\ell_\infty(\N, A)$ and that $[y_i]=[x_i]=\phi(t\otimes f_{i,1})$. Also, using the orthogonality of the elements $h_i$ we get that $y_{i,n}y_{j,n}=0$ for all $i, j\ge 2$, and that $y_{i,n}^*y_{j,n}=0$ for all $i,j\ge 2$ with $i\neq j$. It follows now by \cite[Proposition 3.3.1]{Loring} that there exists a *-homomorphism $\psi\colon \M_d(\CC_0(0,1])\to \ell_\infty(\N, A)$ satisfying $\psi(t\otimes f_{i,1})=y_i$ for all $i$. Therefore, we have
\begin{align*}
&\psi(t\otimes f_{1,1})=(y_1^*y_1)^\frac{1}{2}=((h_{1}^2-9\epsilon^{-\frac 1 2})_+)^\frac{1}{2}\in \CC^*(h_1),\\
&\psi(t\otimes f_{i,i})\le (y_iy_i^*)^{\frac 1 2}\le (h_{i}^2)^\frac 1 2=h_i,
\end{align*}
for all $i$. (In order to obtain the last inequality we are using \cite[Proposition 1.3.8]{PedersenBook}.)
\end{proof}

\begin{theorem}\label{thm: local approx}
Let $A$ be a C*-algebra and let $G$ be a finite group. Let $\alpha\colon G\to \Aut(A)$ be an action with the  Rokhlin property. Then for any finite subset $F\subseteq A\rtimes_\alpha G$ and any $\epsilon>0$ there exist a positive element $a\in A$ and a *-homomorphism $\varphi\colon \M_{|G|}(\overline{aAa})\to A\rtimes_{\alpha}G$ such that $\dist(x, \IIm (\varphi))<\epsilon$ for all $x\in F$. 

\end{theorem}
\begin{proof}
Let $\alpha\colon G\to \Aut(A)$ be an action with the  Rokhlin property and let $(u_g)_{g\in G}\subset \M(A\rtimes_\alpha G)$ be the canonical unitaries satisfying $\alpha_g(a)=u_gau_g^*$ for all $a\in A$ and $g\in G$. Let $F$ be a finite subset of $A\rtimes_\alpha G$. Since $A\rtimes_\alpha G$ is generated by the elements of the set $\{au_g: a\in A, g\in G\}$, we may assume that $F=\{au_g : a\in F_1, g\in G\}$, where $F_1$ is a finite subset of $A$ such that $\alpha_g(F_1)=F_1$ for all $g\in G$.

By (ii) of Proposition \ref{prop: Rokhlin equivalence} there exist orthogonal positive contractions $(r_g)_{g\in G}\subset A^{\infty}\cap F_1'$ satisfying 
\begin{align}\label{eq: r}
\alpha_g(r_h)=r_{gh},\quad (\sum_{g\in G}r_g)b=b\sum_{g\in G}r_g=b,
\end{align}
for $g,h\in G$ and $b\in F_1$. Consider $A^\infty$ as a C*-subalgebra of $(A\rtimes_\alpha G)^\infty$ and set
$x_g:=u_gr_e$ for $g\in G$. Then $x_g\in (A\rtimes_\alpha G)^\infty$, $\|x_g\|\le 1$, and the following orthogonality relations hold
$$x_gx_h=0\quad \forall g,h\neq e, \quad x_g^*x_h=\delta_{g,h} r_e^2\quad \forall g,h\in G.$$
This implies by \cite[Proposition 3.3.1]{Loring} that there exists a *-homomorphism $\phi\colon \M_{|G|}(\CC_0(0,1])\to (A\rtimes_\alpha G)^\infty$ such that $\phi(t\otimes f_{g,e})=x_{g}$. Let $s_g=(s_{g,n})\in \ell_\infty(\N, A)$, with $g\in G$, be orthogonal positive contractions that lift $(r_g)_{g\in G}$ (they exist by \cite[Lemma 10.1.12]{Loring}). Then by Proposition \ref{prop: lifting} the map $\phi$ can be lifted to a *-homomorphism $\psi=(\psi_n)\colon \M_{|G|}(\CC_0(0,1])\to \ell_\infty(\N, A\rtimes_\alpha G)$. Moreover, $\psi$ can be taken such that $\psi(t\otimes f_{e,e})\in \CC^*(s_e)\subseteq \ell_\infty(\N, A)$. Note that $\overline{\psi_n(t\otimes f_{e,e})A\psi_n(t\otimes f_{e,e})}$ is a hereditary C*-subalgebra of $A$. Also, using the equations in \eqref{eq: r} we get 
\begin{align}\label{eq: aupsi}
\begin{aligned}
&\lim_{n\to \infty}\psi_n(t\otimes 1_{\M_{|G|}})au_g=\lim_{n\to \infty}\sum_{h\in G}\psi_n(t\otimes f_{h,h})au_g=au_g,\\
&\lim_{n\to \infty}au_g\psi_n(t\otimes 1_{\M_{|G|}})=\lim_{n\to \infty}au_g\sum_{h\in G}\psi_n(t\otimes f_{h,h})=\lim_{n\to \infty}a(\sum_{h\in G}\psi_n(t\otimes f_{h,h}))u_g=au_g,
\end{aligned}
\end{align}
for $au_g\in F$.

Let us now use that $(r_g)_{g\in G}$ are orthogonal elements of $A^\infty\cap F_1'$ satisfying $\alpha_g(r_h)=r_{gh}$ for all $g,h\in G$. For each $au_k\in F$ we have
\begin{align*}
&\phi(t\otimes f_{e,g})(a u_k)\phi(t\otimes f_{h,e})=x_g^*(au_k)x_h=r_eu_g^*au_ku_hr_e=u_g\alpha_g(r_e)au_ku_hr_e=u_g^*r_gau_{kh}r_e\\
&=u_g^*r_ga\alpha_{kh}(r_e)u_{kh}=u_g^*r_gar_{kh}u_{kh}=u_g^*ar_gr_{kh}u_{kh}
=\delta_{g,kh}u_g^*ar_g^2u_g=\delta_{g,kh}u_g^*au_g\alpha_{g^{-1}}(r^2_g)\\
&=\delta_{g,kh}\alpha_{g^{-1}}(a)r_e^2=\delta_{g,kh}r_e\alpha_{g^{-1}}(a)r_e.
\end{align*}
Hence, since $\psi$ is a lift of $\phi$ we get 
\begin{align}\label{eq: psi}
\lim_{n\to \infty}\|\psi_n(t\otimes f_{e,g})(a u_k)\psi_n(t\otimes f_{h,e})-\delta_{g,kh}\psi_n(t\otimes f_{e,e})\alpha_{g^{-1}}(a)\psi_n(t\otimes f_{e,e})\|= 0.
\end{align}

Set $c_{g,n}=\psi_n(t\otimes f_{e,g})$, $d_n=\psi_n(t\otimes 1_{\M_{|G|}})$, $c_g=(c_{g,n})_{n\in \N}$, and $d=(d_n)_{n\in \N}$. Then by Lemma \ref{lem: matrix units} there exist *-isomorphisms $\rho_n\colon \M_{|G|}(\overline{c_{e,n}(A\rtimes_\alpha G)c_{e,n}})\to \overline{d_n(A\rtimes_\alpha G)d_n}$ satisfying
$$\rho_n((c_{e,n}xc_{e,n})\otimes f_{g,h}):=\psi_n(t\otimes f_{g,e})x\psi_n(t\otimes f_{e,h})=c_{g,n}^*xc_{h,n}.$$
Denote by $\widetilde\rho_n\colon \M_{|G|}(\overline{c_{e,n}Ac_{e,n}})\to \overline{d_n(A\rtimes_\alpha G)d_n}$ the restriction of $\rho_n$ to the C*-subalgebra $\M_{|G|}(\overline{c_{e,n}Ac_{e,n}})$.

Let $\epsilon>0$. By equations \eqref{eq: aupsi} and \eqref{eq: psi} we can choose $n$ large enough such that 
\begin{align*}
&\|d_n^2(a u_k)d_n^2-au_k\|<\frac{\epsilon}{|G|^2+1},\\
&\|c_{g,n}(a u_k)c_{h,n}^*-\delta_{g,kh}c_{e,n}\alpha_{g^{-1}}(a)c_{e,n})\|<\frac{\epsilon}{|G|^2+1},
\end{align*}
for all $au_k\in F$ and $g,h\in G$. It follows that
\begin{align*}
\|\widetilde\rho_n(\sum_{g\in G}(c_{e,n}\alpha_{g^{-1}}(a)c_{e,n})\otimes f_{g,k^{-1}g})-au_k\|&=\|\sum_{g,h\in G}\delta_{g, kh}\rho_n((c_{e,n}\alpha_{g^{-1}}(a)c_{e,n})\otimes f_{g,h})-au_k\|\\
&\le \frac{|G|^2}{|G|^2+1}\epsilon+\|\sum_{g,h\in G}\rho_n((c_{g,n} a u_k c_{h,n}^*)\otimes f_{g,h})-au_k\|\\
&=\frac{|G|^2}{|G|^2+1}\epsilon+\|\sum_{g,h\in G}(c_{g,n}^*c_{g,n} a u_k c_{h,n}^*c_{h,n})-au_k\|\\
&=\frac{|G|^2}{|G|^2+1}\epsilon+\|d_n^2 au_k d_n^2-au_k\|\\
&<\epsilon.
\end{align*}
This shows that $\dist(x, \mathrm{Im}(\widetilde \rho))<\epsilon$ for all $x\in F$.
\end{proof}

\section{Permanence properties}
In this section we show that several classes of C*-algebras are closed under crossed products by finite group actions with the Rokhlin property.

\begin{lemma}\label{lem: positive elements}
Let $A$ and let $B$ be C*-subalgebra of $A$. Let $0<\epsilon<1$. Suppose that $a\in A$ is a positive contraction and $b\in B$ is such that $\|b-a\|<\epsilon$. Then 
$$\|b-(a^*a)^\frac{1}{2}\|<2\epsilon.$$
\end{lemma}
\begin{proof}
By the triangle inequality we have
$$\|b^2-a^*a\|\le \|b(b-a))\|+\|(b-a)^*a\|\le (1+\|b\|)\epsilon<(2+\epsilon)\epsilon<2\epsilon.$$
Hence, by \cite[Lemma 1]{C-E-S} we get
$$\|b-(a^*a)^\frac{1}{2}\|<\sqrt{2}\epsilon^\frac{1}{2}<2\epsilon.$$
\end{proof}

\begin{definition}\label{def: lapprox}
Let $\mathcal{C}$ be a class of C*-algebras. A C*-algebra $A$ is called a {\it local $\mathcal C$-algebra} if for every finite subset $F\subset A$ and every $\epsilon>0$ there exist a C*-algebra $B$ in $\mathcal C$ and a *-homomorphism $\phi\colon B\to A$ such that $\dist(x, \phi(B))<\epsilon$ for all $x\in A$. When $A$ is a unital C*-algebra and the C*-algebras in $\mathcal{C}$ and the *-homomorphisms $\phi$ are unital the C*-algebra $A$ is called {\it unital local $\mathcal C$-algebra}. We say that $\mathcal{C}$ is {\it closed under local approximation} if every local $\mathcal C$-algebra (separable if the C*-algebras in $\mathcal C$ are separable) belongs to $\mathcal{C}$.
\end{definition}

Note that any C*-algebra that is a direct limit of C*-algebras in $\mathcal{C}$ is a local $\mathcal{C}$-algebra. Also, if $\widetilde{\mathcal{C}}$ denotes the class of C*-algebras consisting of the unitization of the C*-algebras in $\mathcal C$ and $\widetilde A$ denotes the unitization of $A$, then $\widetilde A$ is a unital local $\widetilde{\mathcal C}$-algebra whenever $A$ is a local $\mathcal{C}$-algebra. Similarly, if $\mathcal{C}^s$ denotes the class of C*-algebras consisting of the stabilization of the C*-algebras in $\mathcal C$, then $A\otimes \K$ is a local $\mathcal{C}^s$-algebra.

\begin{lemma}\label{lem: ideals}
Let $\mathcal{C}$ be a class of C*-algebras such that if $B$ is a C*-algebra in $\mathcal C$ and $J$ is a closed two-sided ideal of $B$ then $J$ is in $\mathcal C$. Let $A$ be a local $\mathcal C$-algebra and let $I$ be a closed two-sided ideal of $A$. Then $I$ is a local $\mathcal C$-algebra.  
\end{lemma}
\begin{proof}
Let $\mathcal C$, $A$, and $I$ be as in the statement of the lemma. Let $\epsilon>0$ and let $F=\{a_1, a_2, \cdots, a_n\}$ be a finite subset of $I$. Since every element of $I$ is a linear combination of four positive contractions we may assume that the elements of $F$ are positive contractions. Using Lemma \ref{lem: positive elements} and that $A$ is a local $\mathcal C$-algebra we may choose a C*-algebra $B\in \mathcal C$, a *-homomorphism $\phi\colon B\to A$ and positive elements $b_1, b_2, \cdots, b_n$ such that 
$$\|a_i-\phi(b_i)\|<\epsilon,$$
for  all $1\le i\le n$. 
By \cite[Lemma 2.2]{Kirchberg-Rordam} there exist elements $d_i\in A$, with $1\le i\le n$, such that 
$$\phi((b_i-\epsilon)_+)=d_i^*a_id_i\in I,$$
for  all $1\le i\le n$. Let $J$ denote the preimage by $\phi$ of $I$. Then $J$ is a closed two-sided ideal of $B$ and thus, it belongs to $\mathcal C$. Also, by the previous equalities we have that $(b_i-\epsilon)_+\in J$. The lemma now follows using that
$$\|a_i-\phi((b_i-\epsilon)_+)\|\le \|a_i-\phi(b_i)\|+\|\phi(b_i)-\phi((b_i-\epsilon)_+)\|<\epsilon+\epsilon=2\epsilon,$$
for all $1\le i\le n$.
\end{proof}

Recall that a 1-dimensional noncommutative CW-complex (\cite[Definition 2.4.1]{E-L-P}), or shortly a 1-dimensional NCCW-complex, is a C*-algebra $A$ that can be written as a pullback diagram of the form
\begin{equation*}
\xymatrix{A\ar[d]\ar[r] & E\ar[d]^\phi \\ \CC([0,1],F)\ar[r]^\rho & F\oplus F}
\end{equation*}
where $E,F$ are finite dimensional C*-algebras, $\phi$ is an arbitrary $\ast$-homomorphism, and $\rho$ is given by evaluation at $0$ and $1$. In contrast to \cite[Definition 2.4.1]{E-L-P}, the maps $\phi$ and $\rho$ are not assumed to be unital. Special cases of 1-dimensional NCCW-complexes are interval algebras (i.e., algebras of the form $\mathrm C([0,1], E)$, where $E$ is a finite dimensional C*-algebra), and circle algebras (i.e., algebras of the form $\mathrm C(\mathbb T, E)$, where $E$ is a finite dimensional C*-algebra). Direct limits of sequences of interval algebras are called AI-algebras whereas direct limits of sequences of circles algebras are called AT-algebras.

\begin{proposition}\label{prop: local approximation}
The following classes of C*-algebras are closed under local approximation:
\begin{itemize}
\item[(i)] Purely infinite C*-algebras;
\item[(ii)] C*-algebras of stable rank one;
\item[(iii)] C*-algebras of real rank zero;
\item[(iv)] C*-algebras of nuclear dimension at most $n$, where $n\in \mathbb{Z}_+$;
\item[(v)] Separable $\mathcal D$-stable C*-algebras, where $\mathcal D$ is a strongly self-absorbing C*-algebra;
\item[(vi)] Simple C*-algebras;
\item[(vii)] Separable C*-algebras whose quotients are stably projectionless;
\item[(viii)] Simple stably projectionless C*-algebras; 
\item[(ix)] C*-algebras that are stably isomorphic to direct limits of sequences of C*-algebras in a class $\mathcal S$, where $\mathcal S$ is a class of finitely generated semiprojective C*-algebras that is closed under taking tensor products by matrix algebras over $\C$.
\item[(x)] Separable AF-algebras;
\item[(xi)] C*-algebras that are stable isomorphic to AI-algebras;
\item[(xii)] C*-algebras that are stable isomorphic AT-algebras;
\item[(xiii)] C*-algebras that can be expressed as a sequential direct limit of 1-dimensional NCCW-complexes;
\item[(xiv)] Separable C*-algebras with almost unperforated Cuntz semigroup;
\item[(xv)] Simple C*-algebras with strict comparison of positive elements.
\item[(xvi)] Separable C*-algebras whose closed two-sided ideals are nuclear and satisfy the Universal Coefficient Theorem.
\end{itemize}
\end{proposition}
\begin{proof}
In \cite[Proposition 4.18]{Kirchberg-Rordam}, \cite[Theorem 5.1]{Rieffel}, \cite[Proposition 3.1]{Brown-Pedersen}, and \cite[Proposition 2.3]{W-Z} it is shown that the classes of C*-algebras described in first four parts of the proposition are closed under taking arbitrary direct limits of sequences. The proofs that these classes of C*-algebras are closed under local approximation are essentially the same; thus, we will omit them.

(v) Let $\mathcal{C}$ be the class of  $\mathcal D$-stable C*-algebras and let $A$ be a local $\mathcal C$-algebra. Let $F\subset A$ and $G\subset \mathcal D$ be finite subsets and let $\epsilon>0$. By Lemma \ref{lem: Dstability} it is enough to show that there exists a completely positive contractive map $\rho\colon \mathcal{D}\to A$ satisfying (i), (ii), and (iii) of Lemma \ref{lem: Dstability}. Without loss of generality we may assume that the elements of $G$ are contractions.

Given that $A$ is a local $\mathcal C$-algebra, there exist a $\mathcal D$-stable C*-algebra $B$ and a *-homomorphism $\phi\colon B\to A$ such that $\dist(a, \phi(B))<\frac{\epsilon}{3}$, for all $a\in F$. Now we can choose a finite subset $\widetilde F\subset B$ satisfying $\dist(a, \phi(\widetilde F))<\frac \epsilon 3$ for all $a\in F$. By Lemma \ref{lem: Dstability} applied to the finite subsets $\widetilde F\subset B$ and $G\subset \mathcal D$, and to $\frac \epsilon 3$, there exists a completely positive contractive map $\psi\colon \mathcal{D}\to B$ such that
\begin{align}\label{eq: cpc}
\|b\psi(1_{\mathcal D})-b\|<\frac{\epsilon}{3}, \quad \|b\psi(d)-\psi(d)b\|<\frac{\epsilon}{3},\quad \|b(\phi(dd')-\phi(d)\phi(d'))\|<\frac{\epsilon}{3},
\end{align}
for all $b\in \widetilde F$ and $d, d'\in G$. Put $\rho=\phi\circ \psi$. Then $\rho$ is a completely positive contraction. For each $a\in F$ choose $b\in \widetilde F$ such that $\|a-\phi(b)\|<\frac{\epsilon}{3}$. Then using the triangle inequality and the inequalities in \eqref{eq: cpc} we get
\begin{align*}
&\|a\rho(1_{\mathcal D})-a\|\le \|(a-\phi(b))\rho(1_{\mathcal D})\|+\|\phi(b-\psi(1_{\mathcal D})b)\|+\|\phi(b)-a\|<\epsilon,\\
& \|a\rho-\rho(d)a\|\le \|(a-\phi(b))\rho(d)\|+\|\phi(b\psi(d)-\psi(d)b)\|+\|\rho(d)(a-\phi(b))\|<\epsilon,\\
& \|a(\rho(dd')-\rho(d)\rho(d'))\|\le \|(a-\phi(b))(\rho(dd')-\rho(d)\rho(d'))\|+\|\phi(b(\psi(dd')-\psi(d)\psi(d')))\|<\epsilon,
\end{align*}
for all $d, d'\in G$.
This implies that $\rho$ satisfies (i), (ii), and (iii) of Lemma \ref{lem: Dstability}. Therefore, $A$ is $\mathcal{D}$-stable by Lemma \ref{lem: Dstability}.

(vi) Let $\mathcal{C}$ be the class of simple C*-algebras and let $A$ be a local $\mathcal{C}$-algebra. Let $a_1, a_2\in A$ be positive elements of norm one and let $0<\epsilon<\frac 1 2$. Then there exist a simple C*-algebra $B$ and a *-homomorphism $\phi\colon B\to A$ such that $\dist (a_i, \phi(B))<\epsilon$ for $i=1,2$. Choose elements $b_1, b_2\in B$ such that $\|a_i-\phi(b_i)\|<\epsilon$ for $i=1,2$. Note that by Lemma \ref{lem: positive elements} we may assume that $b_1$ and $b_2$ are positive. Also, $(\phi(b_2)-\epsilon)_+$ is nonzero as $\epsilon<\frac 1 2$ and $\|a_2\|=1$.
By \cite[Lemma 2.2]{Kirchberg-Rordam} there exist $d_i\in A$, with $i=1,2$, such that 
$$(a_1-\epsilon)_+=d_1^*\phi(b_1)d_1,\quad (\phi(b_2)-\epsilon)_+=d_1^*a_2d_1.$$
This implies that $[(a_1-\epsilon)_+]\le [\phi(b_1)]$ and $[(\phi(b_2)-\epsilon)_+]\le [a_2]$ in $\Cu(A)$. Since $B$ is simple, $[b_1]\le \sup_n n[(b_2-\epsilon)_+]$ in $\Cu(B)$. Hence,
$$[(a_1-\epsilon)_+]\le [\phi(b_1)]=\Cu(\phi)([b_1])\le\sup_n n\Cu(\phi)([(b_2-\epsilon)_+])=\sup_n n[(\phi(b_2)-\epsilon)_+]\le \sup_n n[a_2],$$
in $\Cu(A)$.
This implies that $(a_1-\epsilon)_+$ is in the closed two-sided ideal generated by $a_2$. Since this holds for every $\epsilon>0$ and $a=\lim_{\epsilon\to 0}(a_1-\epsilon)_+$, $a_1$ is also an element of this ideal. 
Using now that $a_1$ and $a_2$ are arbitrary we get that $A$ is simple.

(vii) Let $\mathcal{C}$ be the class of separable C*-algebras $B$ such that for every closed two-sided ideal $I$ of $B$, $B/I\otimes \K$ is projectioneless.  Let $A$ be a local $\mathcal C$-algebra. By the remarks after Definition \ref{def: lapprox}, and \cite[Corollary 2.3 part (i)]{Rordam} we may assume that $A$ and all the algebras in $\mathcal C$ are stable. Let $I$ be an ideal of $A$ and let $\pi\colon A\to A/I$ denote the quotient map. Suppose that there exists a projection $p\in A/I$ with $p\neq 0$. Choose a positive element $a\in A$ such that $\pi(a)=p$. Since $A$ is a local $\mathcal C$-algebra, there exist a C*-algebra $B\in \mathcal{C}$, a *-homomorphism $\phi\colon B\to A$, a positive element $a'\in A$ (here we are also using Lemma \ref{lem: positive elements}) such that
$$\|\phi(b)-a\|<\frac 1 4.$$
By passing to the quotient of $B$ by the kernel of $\phi$ and using that this quotient is projectionless we may assume that $\phi$ is injective.
Put $J:=\phi^{-1}(I)$. Then $J$ is a closed two-sided ideal of $B$ and $\phi$ induces an injective *-homomorphism $\overline\phi\colon B/J\to A/I$. Let $b'$ be the image of $b$ in $B/J$. Then using that $\|\phi(b)-a\|<\frac 1 4$ we get 
$$\|\overline\phi(b')-p\|<\frac 1 4.$$
This implies that $\overline\phi(b')$ has gap in its spectrum by \cite[Lemma 2.2.3]{Rordam Book}. Hence, by functional calculus $\overline\phi(B/J)$ contains a nonzero projection. Using that $\overline\phi$ is injective we get that $B/J$ contains a nonzero projection, which contradicts the fact that $B/J$ is projectionless. Therefore, $A$ belongs to $\mathcal C$.

Note that the proof above (with some minor modifications) also works for the class of (not necessarily separable) simple stably projectionless C*-algebras. Hence, the class in (viii) is closed under local approximation.

(ix) Let $\mathcal{C}$ and $\mathcal{S}$ be the classes of C*-algebras described in (ix). Let $A$ be a separable local $\mathcal C$-algebra. It suffices to show that $A\otimes \K$ can be expressed as a direct limit of a sequence of C*-algebras in $\mathcal S$. Since $A$ is a local $\mathcal C$-algebra, $A\otimes \K$ is a local $\mathcal S^s$-algebra, where $\mathcal S^s=\{B\otimes \K: B\in \mathcal S\}$. It follows that $A\otimes \K$ can be locally approximated by C*-algebras of the form $\M_n(B)$, where $B\in \mathcal S$ and $n\in\N$ (note that by \cite[Theorem 14.2.2]{Loring} each of the algebras $\M_n(B)$ is semiprojective). This implies that $A\otimes \K$ is a $\mathcal S$-algebra, given that $\mathcal S$ is closed by taking tensor products with matrix algebras over $\C$. It follows now by \cite[Lemma 2.2.5]{E-L-P} and \cite[Lemma 15.2.2]{Loring} that $A\otimes \K$ is a direct limit of a sequence of C*-algebras in $\mathcal S$.

Parts (x), (xi), (xii), and (xiii) are a consequence of (ix), by taking $\mathcal S$ to be the class of finite dimensional C*-algebras, the class of interval algebras, the class of circle algebras, and the class of 1-dimensional NCCW-complexes, respectively. Each of these algebras are semiprojective and finitely generated by \cite[Lemma 2.4.3 and Theorem 6.2.2]{E-L-P}. The fact that separable AF-algebras are closed under local approximation was first proved in \cite[Theorem 2.2]{Bratelli} using a different method.


(xiv) Let $\mathcal{C}$ be the class of separable C*-algebras with almost unperforated Cuntz semigroup. Let $A$ be a local $\mathcal C$-algebra. Note that if $\mathcal{C}^s$ denotes the class of C*-algebras consisting of the stabilization of the algebras in $\mathcal C$, then $A\otimes \K$ is a local $\mathcal C^s$-algebra. Therefore, using that the Cuntz semigroup is a stable functor, we may assume that $A$ and all the algebras in $\mathcal C$ are stable. Also, using that $\mathcal{C}$ is closed under taking quotients by Lemma \ref{lem: almost quotient}, we may assume that the *-homomorphisms used in the approximation of $A$ by algebras in $\mathcal{C}$ are injective (here we are using that all the algebras in $\mathcal C$ are separable; thus, all their closed two-sided ideal are $\sigma$-unital).

Since $A$ is stable, it is enough to show that, for $a,b\in A$, if $(n+1)[a]\le n[b]$ then $[a]\le [b]$. Let $a,b\in A_+$ be such elements. Then it follows that
$$a\otimes 1_{\M_{n+1}}\precsim b\otimes 1_{\M_n}.$$
Let $\epsilon>0$. Then by (ii) of Lemma \ref{lem: Cuntz subequivalence} there exists $\delta>0$ such that
$$(a-\epsilon)_+\otimes 1_{\M_{n+1}}\precsim (b-\delta)_+\otimes 1_{\M_n}.$$
By the definition of Cuntz subequivalence, there exists $d\in \M_{n+1}(A)$ such that 
$$\|(a-\epsilon)_+\otimes 1_{\M_{n+1}}-d^*((b-\delta)_+\otimes 1_{\M_n})d\|<\epsilon.$$
Using now that $A$ is a local $\mathcal C$-algebra, there exist a C*-algebra $B\in \mathcal{C}$, an injective *-homomorphism $\phi\colon B\to A$, elements $a', b'\in B_+$ (here we are also using Lemma \ref{lem: positive elements}) and an element $d'\in \M_{n+1}(B)$ such that 
\begin{align}\label{eq: aphi}
\begin{aligned}
&\|a-\phi(a')\|<\epsilon, \quad \|b-\phi(b')\|<\delta, \\
&\|(\phi\otimes 1_{\M_{n+1}})\left((a'-\epsilon)_+\otimes 1_{\M_{n+1}}-(d')^*((b'-\delta)_+\otimes 1_{\M_n})d'\right)\|<\epsilon.
\end{aligned}
\end{align}
Since $\phi$ is injective, it is isometric. Hence, by the last inequality we have
$$\|(a'-\epsilon)_+\otimes 1_{\M_{n+1}}-(d')^*((b'-\delta)_+\otimes 1_{\M_n})d'\|<\epsilon,$$
in $\M_{n+1}(B)$.
It follows that 
$$(a'-2\epsilon)_+\otimes 1_{\M_{n+1}}\precsim(b'-\delta)_+\otimes 1_{\M_n}$$
by (i) of Lemma \ref{lem: Cuntz subequivalence}. Hence, $(n+1)[(a'-2\epsilon)_+]\le n[(b'-\delta)_+]$ in $\Cu(B)$. Using now that $\Cu(B)$ is almost unperforated we get $[(a'-2\epsilon)_+]\le[(b'-\delta)_+]$. By the first two inequalities in \eqref{eq: aphi} and (i) of Lemma \ref{lem: Cuntz subequivalence} we get
$$(a-3\epsilon)_+\precsim (\phi(a')-2\epsilon)_+, \quad (\phi(b')-\delta)_+\precsim b.$$
Therefore,
$$[(a-3\epsilon)_+]\le [(\phi(a')-2\epsilon)_+]=[\phi((a'-2\epsilon)_+)]\le[\phi((b'-\delta)_+)]=[(\phi(b')-\delta)_+)]\le[b]. $$
Using now that $[a]=\sup_{\epsilon>0}[(a-3\epsilon)_+]$ we get $[a]\le [b]$, as desired.

(xv) Let $\mathcal C$ be the class of simple C*-algebras with strict comparison of positive elements and let $A$ be a local $\mathcal C$-algebra. Then by Lemma \ref{lem: strict-almost} every C*-algebra in $\mathcal C$ has almost unperforated Cuntz semigroup. Note now that the proof of (xii) also applies  for the class $\mathcal C$ and the algebra $A$. The separability condition was only used to assure that the *-homomorphisms used in the local approximation of $A$ are injective, and this is obviously true for simple C*-algebras. Hence, it follows that $A$ has almost unperforated Cuntz semigroup. Therefore, $A$ has strict comparison of positive elements by Lemma \ref{lem: strict-almost}. 

(xvi)  Let $\mathcal C$ be the class of separable C*-algebras whose closed two-sided ideals are nuclear and satisfy the Universal Coefficient Therem (shortly, the UCT). Let $A$ be a local $\mathcal C$-algebra and let $I$ be a closed two-sided ideal of $A$. Then $I$ is a local $\mathcal C$-algebra by Lemma \ref{lem: ideals}. It follows now by \cite[Theorem 1.1]{DadarlatUCT} that $I$ satisfies the UCT. That $I$ is nuclear follows using Arveson's extension theorem and that $I$ can be locally approximated by nuclear C*-algebras.
\end{proof}

\begin{theorem}\label{thm: permanence}
The following classes of C*-algebras are preserved under taking crossed products with actions of finite groups with the  Rokhlin property:
\begin{itemize}
\item[(i)] Purely infinite C*-algebras;
\item[(ii)] C*-algebras of stable rank one;
\item[(iii)] C*-algebras of real rank zero;
\item[(iv)] C*-algebras of nuclear dimension at most $n$, where $n\in \mathbb{Z}_+$;
\item[(v)] Separable $\mathcal D$-stable C*-algebras, where $\mathcal D$ is a strongly self-absorbing C*-algebra;
\item[(vi)] Simple C*-algebras;
\item[(vii)] Simple C*-algebras that are stably isomorphic to direct limits of sequences of C*-algebras in a class $\mathcal S$, where $\mathcal S$ is a class of finitely generated semiprojective C*-algebras that is closed under taking tensor products by matrix algebras over $\C$;
\item[(viii)] Separable AF-algebras;
\item[(ix)] Separable simple C*-algebras that are stably isomorphic to AI-algebras;
\item[(x)] Separable simple C*-algebras that are stably isomorphic to AT-algebras;
\item[(xi)] C*-algebras that are stably isomorphic to sequential direct limits of one-dimensional NCCW-complexes;
\item[(xii)] Separable C*-algebras whose quotients are stably projectionless;
\item[(xiii)] Simple stably projectionless C*-algebras; 
\item[(xiv)] Separable C*-algebras with almost unperforated Cuntz semigroup;
\item[(xv)] Simple C*-algebras with strict comparison of positive elements.
\item[(xvi)] Separable C*-algebras whose closed two-sided ideals are nuclear and satisfy the Universal Coefficient Theorem.
\end{itemize}
\end{theorem}
\begin{proof} 
Let $A$ be a C*-algebra and let $\alpha\colon G\to \Aut(A)$ be an action of a finite group $G$ on $A$ with the Rokhlin property. Let $\mathcal S$ be a class of C*-algebras such that for $B\in \mathcal S$, $n\in\N$, and $b\in B_+$ one has $\M_n(A)\in\mathcal S$ and $\overline{bBb}\in\mathcal S$. Then by Theorem \ref{thm: local approx}, if $A$ is a local $\mathcal S$-algebra then so is the crossed product C*-algebra $A\rtimes_\alpha G$. Given that all the classes of C*-algebras described in the theorem are closed under local approximation by Proposition \ref{prop: local approximation} and under tensor product by matrix algebras, it is sufficient to show that they are closed under taking $\sigma$-unital hereditary subalgebras. This clearly holds for the classes of C*-algebras in (vi), (vii), (ix), (x), (xii), and (xiii) (here we are using \cite[Theorem 2.8]{Brown}). For the classes in (i), (iii), (iv), (v), (viii), and (xvi) this follows by \cite[Proposition 3.15(iii)]{Kirchberg-Rordam}, \cite[Corollary 2.8]{Brown-Pedersen}, \cite[Propostion 2.5]{W-Z}, \cite[Corollary 3.1]{Toms-Winter}, \cite[Theorem 3.1]{Elliott-Ideals}, and \cite[Propositions 2.1.2(ii) and 2.4.7(iii)]{RordamBook}, respectively. For the class of C*-algebras in (xiv) this follows using \cite[Theorem 2.8]{Brown} and \cite[Proposition 2.3]{Hishberg-Rordam-Winter}. For the class of C*-algebras in (xv) this holds by (xiv) and \cite[Lemma 6.1]{Tikuisis-Toms}. Finally, for the class of algebras in (xi) this follows by the following argument. Let $B$ be a hereditary C*-subalgebra of a C*-algebra $A$ that is stably isomorphic to a direct limit of a sequence of 1-dimensional NCCW-complexes. Then $B\otimes \K$ is isomorphic to an ideal of $A\otimes \K$ by \cite[Theorem 2.8]{Brown}. It follows that $B\otimes \K$ is the direct limit of a sequence of ideals of 1-dimensional NCCW-complexes. Each of these ideals is a direct limit of a sequence of 1-dimensional NCCW-complexes by \cite[Lemma 3.11]{Santiago}. Hence, $B\otimes \K$ can be locally approximated by 1-dimensional NCCW-complexes. This implies by Proposition \ref{prop: local approximation} that $B\otimes \K$ is a direct limit of a sequence of such algebras. Therefore, $B$ belongs to the class of C*-algebras in (xi). 
\end{proof}


\begin{bibdiv}
\begin{biblist}


\bib{Antoine-Perera-Santiago}{article}{
   author={Antoine, R.},
   author={Perera, F.},
   author={Santiago, L.},
   title={Pullbacks, $C(X)$-algebras, and their Cuntz semigroup},
   journal={J. Funct. Anal.},
   volume={260},
   date={2011},
   number={10},
   pages={2844--2880},
}

\bib{Bratelli}{article}{
   author={Bratteli, O.},
   title={Inductive limits of finite dimensional $C\sp{\ast} $-algebras},
   journal={Trans. Amer. Math. Soc.},
   volume={171},
   date={1972},
   pages={195--234},
}

\bib{Brown}{article}{
   author={Brown, L. G.},
   title={Stable isomorphism of hereditary subalgebras of $C\sp*$-algebras},
   journal={Pacific J. Math.},
   volume={71},
   date={1977},
   number={2},
   pages={335--348},
}


\bib{Brown-Pedersen}{article}{
   author={Brown, L. G.},
   author={Pedersen, G. K.},
   title={$C\sp *$-algebras of real rank zero},
   journal={J. Funct. Anal.},
   volume={99},
   date={1991},
   number={1},
   pages={131--149},
}

\bib{B-R-T-T-W}{article}{
   author={Blackadar, B.},
   author={Robert, L.},
   author={Tikuisis, A. P.},
   author={Toms, A. S.},
   author={Winter, W.},
   title={An algebraic approach to the radius of comparison},
   journal={Trans. Amer. Math. Soc.},
   volume={364},
   date={2012},
   number={7},
   pages={3657--3674},
}

\bib{C-R-S}{article}{
   author={Ciuperca, A.},
   author={Robert, L.},
   author={Santiago, L.},
   title={The Cuntz semigroup of ideals and quotients and a generalized
   Kasparov stabilization theorem},
   journal={J. Operator Theory},
   volume={64},
   date={2010},
   number={1},
   pages={155--169},
}

\bib{C-E-S}{article}{
   author={Ciuperca, A.},
   author={Elliott, G. A.},
   author={Santiago, L.},
   title={On inductive limits of type-I $C\sp *$-algebras with
   one-dimensional spectrum},
   journal={Int. Math. Res. Not. IMRN},
   date={2011},
   number={11},
   pages={2577--2615},
}





\bib{Coward-Elliott-Ivanescu}{article}{
   author={Coward, K. T.},
   author={Elliott, G. A.},
   author={Ivanescu, C.},
   title={The Cuntz semigroup as an invariant for C*-algebras},
   journal={J. Reine Angew. Math.},
   volume={623},
   date={2008},
   pages={161--193},
}

\bib{DadarlatUCT}{article}{
   author={Dadarlat, M.},
   title={Some remarks on the universal coefficient theorem in $KK$-theory},
   conference={
      title={Operator algebras and mathematical physics},
      address={Constan\c ta},
      date={2001},
   },
   book={
      publisher={Theta, Bucharest},
   },
   date={2003},
   pages={65--74},
}



\bib{E-L-P}{article}{
   author={Eilers, S.},
   author={Loring, T. A.},
   author={Pedersen, G. K.},
   title={Stability of anticommutation relations: an application of
   noncommutative CW complexes},
   journal={J. Reine Angew. Math.},
   volume={499},
   date={1998},
   pages={101--143},
}


\bib{Elliott-Ideals}{article}{
   author={Elliott, G. A.},
   title={Automorphisms determined by multipliers on ideals of a
   $C\sp*$-algebra},
   journal={J. Functional Analysis},
   volume={23},
   date={1976},
   number={1},
   pages={1--10},
}


\bib{Fack-Marechal}{article}{
   author={Fack, Th.},
   author={Mar{\'e}chal, O.},
   title={Sur la classification des sym\'etries des $C\sp{\ast} $-algebres
   UHF},
   language={French},
   journal={Canad. J. Math.},
   volume={31},
   date={1979},
   number={3},
   pages={496--523},
}

\bib{Gardella-Santiago}{article}{
   author={Gardella, E.},
   author={Santiago, L.},
   title={Classification of actions of finite groups on C*-algebras},
   journal={preprint},
   volume={},
   date={2014},
   number={},
   pages={},
}

\bib{G-L-P}{article}{
   author={Gootman, E. C.},
   author={Lazar, A. J.},
   author={Peligrad, C.},
   title={Spectra for compact group actions},
   journal={J. Operator Theory},
   volume={31},
   date={1994},
   number={2},
   pages={381--399},
}

\bib{Herman-Jones}{article}{
   author={Herman, R. H.},
   author={Jones, V. F. R.},
   title={Models of finite group actions},
   journal={Math. Scand.},
   volume={52},
   date={1983},
   number={2},
   pages={312--320},
}
\bib{Hishberg-Rordam-Winter}{article}{
   author={Hirshberg, I.},
   author={R{\o}rdam, M.},
   author={Winter, W.},
   title={$\scr C\sb 0(X)$-algebras, stability and strongly self-absorbing
   $C\sp *$-algebras},
   journal={Math. Ann.},
   volume={339},
   date={2007},
   number={3},
   pages={695--732},
}


\bib{Izumi-I}{article}{
   author={Izumi, M.},
   title={Finite group actions on C*-algebras with the Rohlin
   property. I},
   journal={Duke Math. J.},
   volume={122},
   date={2004},
   number={2},
   pages={233--280},
}


\bib{Izumi-II}{article}{
   author={Izumi, M.},
   title={Finite group actions on C*-algebras with the Rohlin
   property. II},
   journal={Adv. Math.},
   volume={184},
   date={2004},
   number={1},
   pages={119--160},
}





\bib{Kirchberg-Rordam}{article}{
  author={Kirchberg, E.},
   author={R{\o}rdam, M.},
   title={Infinite non-simple C*-algebras: absorbing the Cuntz
   algebras $\scr O\sb \infty$},
   journal={Adv. Math.},
  volume={167},
  date={2002},
   number={2},
   pages={195--264},
}

\bib{Kishimoto}{article}{
   author={Kishimoto, A.},
   title={On the fixed point algebra of a UHF algebra under a periodic
   automorphism of product type},
   journal={Publ. Res. Inst. Math. Sci.},
   volume={13},
   date={1977/78},
   number={3},
   pages={777--791},
}






\bib{Loring}{book}{
   author={Loring, T. A.},
   title={Lifting solutions to perturbing problems in $C\sp *$-algebras},
   series={Fields Institute Monographs},
   volume={8},
   publisher={American Mathematical Society},
   place={Providence, RI},
   date={1997},
   pages={x+165},
}

\bib{Nawata}{article}{
   author={Nawata, N.},
   title={Finite group actions on certain stably projectionless C*-algebras with the Rohlin property},
   journal={Preprint, arXiv:1308.0429},
   volume={},
   date={2013},
}

\bib{Osaka-Phillips}{article}{
   author={Osaka, H.},
   author={Phillips, N. C.},
   title={Crossed products by finite group actions with the Rokhlin
   property},
   journal={Math. Z.},
   volume={270},
   date={2012},
   number={1-2},
   pages={19--42},
}


\bib{PedersenBook}{book}{
   author={Pedersen, G. K.},
   title={$C\sp{\ast} $-algebras and their automorphism groups},
   series={London Mathematical Society Monographs},
   volume={14},
   publisher={Academic Press Inc. [Harcourt Brace Jovanovich Publishers]},
   place={London},
   date={1979},
   pages={ix+416},
}


\bib{Pedersen}{article}{
   author={Pedersen, G. K.},
   title={Unitary extensions and polar decompositions in a $C\sp
   \ast$-algebra},
   journal={J. Operator Theory},
   volume={17},
   date={1987},
   number={2},
   pages={357--364},
}



\bib{Phillips-Equivariant-K-theory-and-freeness-of-actions}{book}{
   author={Phillips, N. C.},
   title={Equivariant $K$-theory and freeness of group actions on $C\sp
   *$-algebras},
   series={Lecture Notes in Mathematics},
   volume={1274},
   publisher={Springer-Verlag},
   place={Berlin},
   date={1987},
}


\bib{Phillips-freeness}{article}{
   author={Phillips, N. C.},
   title={Freeness of actions of finite groups on $C\sp *$-algebras},
   conference={
      title={Operator structures and dynamical systems},
   },
   book={
      series={Contemp. Math.},
      volume={503},
      publisher={Amer. Math. Soc.},
      place={Providence, RI},
   },
   date={2009},
   pages={217--257},
}



\bib{Phillips-TracialR}{article}{
   author={Phillips, N. C.},
   title={The tracial Rokhlin property for actions of finite groups on $C\sp
   \ast$-algebras},
   journal={Amer. J. Math.},
   volume={133},
   date={2011},
   number={3},
   pages={581--636},
}

\bib{Phillips-generic}{article}{
   author={Phillips, N. C.},
   title={The tracial Rokhlin property is generic},
   journal={Preprint, arXiv:1209.3859},
   volume={},
   date={2012},
}


\bib{RieffelMoritaeq}{article}{
   author={Rieffel, M. A.},
   title={Morita equivalence for operator algebras},
   conference={
      title={Operator algebras and applications, Part I},
      address={Kingston, Ont.},
      date={1980},
   },
   book={
      series={Proc. Sympos. Pure Math.},
      volume={38},
      publisher={Amer. Math. Soc.},
      place={Providence, R.I.},
   },
   date={1982},
   pages={285--298},
}

\bib{Rieffel}{article}{
   author={Rieffel, M. A.},
   title={Dimension and stable rank in the $K$-theory of
   $C\sp{\ast}$-algebras},
   journal={Proc. London Math. Soc. (3)},
   volume={46},
   date={1983},
   number={2},
   pages={301--333},
}


\bib{Robert-Santiago}{article}{
   author={Robert, L.},
   author={Santiago, L.},
   title={Classification of $C\sp \ast$-homomorphisms from $C\sb 0(0,1]$ to
   a $C\sp \ast$-algebra},
   journal={J. Funct. Anal.},
   volume={258},
   date={2010},
   number={3},
   pages={869--892},
}

\bib{Robert}{article}{
   author={Robert, L.},
   title={Classification of inductive limits of 1-dimensional NCCW
   complexes},
   journal={Adv. Math.},
   volume={231},
   date={2012},
   number={5},
   pages={2802--2836},
}




\bib{RordamBook}{book}{
   author={R{\o}rdam, M.},
   author={Larsen, F.},
   author={Laustsen, N.},
   title={An introduction to $K$-theory for $C\sp *$-algebras},
   series={London Mathematical Society Student Texts},
   volume={49},
   publisher={Cambridge University Press},
   place={Cambridge},
   date={2000},
   pages={xii+242},
}


\bib{RordamBook}{article}{
   author={R{\o}rdam, M.},
   title={Classification of nuclear, simple $C\sp *$-algebras},
   conference={
      title={Classification of nuclear $C\sp *$-algebras. Entropy in
      operator algebras},
   },
   book={
      series={Encyclopaedia Math. Sci.},
      volume={126},
      publisher={Springer},
      place={Berlin},
   },
   date={2002},
   pages={1--145},
}

\bib{Rordam}{article}{
   author={R{\o}rdam, M.},
   title={Stable $C\sp *$-algebras},
   conference={
      title={Operator algebras and applications},
   },
   book={
      series={Adv. Stud. Pure Math.},
      volume={38},
      publisher={Math. Soc. Japan},
      place={Tokyo},
   },
   date={2004},
   pages={177--199},
}




\bib{Santiago}{article}{
   author={Santiago, L.},
   title={Reduction of the dimension of nuclear C*-algebras},
   journal={arXiv:1211.7159},
   volume={},
   date={2012},
}


\bib{Sierakowski}{article}{
   author={Sierakowski, A.},
   title={The ideal structure of reduced crossed products},
   journal={M\"unster J. Math.},
   volume={3},
   date={2010},
   pages={237--261},
}

\bib{Tikuisis-Toms}{article}{
   author={Toms, A. S.},
   author={Tikuisis, A.},
   title={On the structure of the Cuntz semigroup in (possibly) nonunital C*-algebras},
   journal={arXiv:1210.2235},
   volume={},
   date={2012},
}



\bib{Toms-Winter}{article}{
   author={Toms, A. S.},
   author={Winter, W.},
   title={Strongly self-absorbing C*-algebras},
   journal={Trans. Amer. Math. Soc.},
   volume={359},
   date={2007},
   number={8},
   pages={3999--4029},
}


\bib{W-Z}{article}{
   author={Winter, W.},
   author={Zacharias, J.},
   title={The nuclear dimension of $C\sp \ast$-algebras},
   journal={Adv. Math.},
   volume={224},
   date={2010},
   number={2},
   pages={461--498},
}


\bib{WinterSSA}{article}{
   author={Winter, W.},
   title={Strongly self-absorbing $C\sp *$-algebras are $\scr Z$-stable},
   journal={J. Noncommut. Geom.},
   volume={5},
   date={2011},
   number={2},
   pages={253--264},
}


\end{biblist}
\end{bibdiv}
\end{document}